\documentclass[11pt,reqno]{amsart}
\usepackage{amsmath, amsthm, amssymb}
\usepackage{enumitem}
\usepackage{graphicx} 
\usepackage{bbm}

\usepackage{fullpage}
\usepackage{empheq}
\usepackage[many]{tcolorbox}
\usepackage{xcolor}
\usepackage{wasysym}
\usepackage{mathtools}
\usepackage{scalerel}
\usepackage{array}
\usepackage{tikz}
\usetikzlibrary{cd}
\usetikzlibrary{calc}
\usetikzlibrary{decorations,decorations.pathreplacing,decorations.markings,decorations.pathmorphing}
\tikzstyle{snake}=[decorate, decoration={snake, segment length=1mm, amplitude=.5mm}]

\tikzset{super thick/.style={line width=3pt}}
\tikzstyle{far>}=[decoration={markings, mark=at position 0.75 with {\arrow{>}}}, postaction={decorate}]
\tikzstyle{mid>}=[decoration={markings, mark=at position 0.55 with {\arrow{>}}}, postaction={decorate}]
\tikzstyle{mid<}=[decoration={markings, mark=at position 0.55 with {\arrow{<}}}, postaction={decorate}]
\tikzset{super thick/.style={line width=3pt}}
\tikzstyle{far>}=[decoration={markings, mark=at position 0.75 with {\arrow{>}}}, postaction={decorate}]
\tikzstyle{mid>}=[decoration={markings, mark=at position 0.55 with {\arrow{>}}}, postaction={decorate}]
\tikzstyle{mid<}=[decoration={markings, mark=at position 0.55 with {\arrow{<}}}, postaction={decorate}]
\tikzstyle{knot}=[preaction={super thick, white, draw}]
\tikzstyle{coupon}=[draw, very thick, rectangle, rounded corners=5pt]
\tikzset{Rightarrow/.style={double equal sign distance,>={Implies},->},
triplecd/.style={-,preaction={draw,Rightarrow}},
quadruplecd/.style={preaction={draw,Rightarrow,
shorten >=0pt
},
shorten >=1pt,
-,double,double
distance=0.2pt}}
\tikzset{
    tripleline/.style args={[#1] in [#2] in [#3]}{
        #1,preaction={preaction={draw,#3},draw,#2}
    }
}
\tikzstyle{triple}=[tripleline={[line width=.15mm,black] in
      [line width=.7mm,white] in
      [line width=1mm,black]}] 
\tikzset{
    quadrupleline/.style args={[#1] in [#2] in [#3] in [#4]}{
        #1,preaction={preaction={preaction={draw,#4},draw,#3}, draw,#2}
    }
}
\tikzstyle{quadruple}=[quadrupleline={[line width=.3mm,white] in
      [line width=.6mm,black] in
      [line width=1.2mm,white] in
      [line width=1.5mm,black]}]

\definecolor{violet}{RGB}{148,0,211}
\definecolor{DarkGreen}{RGB}{0,150,0}
\definecolor{rufous}{HTML}{A81C07}
\definecolor{boysenberry}{HTML}{873260}
\definecolor{OliveGreen}{HTML}{6D712E}
\definecolor{yellow}{RGB}{200, 120, 60}

\usepackage[plainpages=false,hypertexnames=false,pdfpagelabels]{hyperref}
\definecolor{medium-blue}{rgb}{0,0,.8}
\hypersetup{colorlinks, linkcolor={purple}, citecolor={medium-blue}, urlcolor={medium-blue}}
\newcommand{\arxiv}[1]{\href{http://arxiv.org/abs/#1}{\tt arXiv:\nolinkurl{#1}}}
\newcommand{\arXiv}[1]{\href{http://arxiv.org/abs/#1}{\tt arXiv:\nolinkurl{#1}}}

\newcommand{\doi}[1]{\href{http://dx.doi.org/#1}{{\tt DOI:#1}}}

\DeclareMathOperator{\End}{End}

\DeclareMathOperator{\ev}{ev}

\DeclareMathOperator{\Hom}{Hom}
\DeclareMathOperator{\id}{id}

\DeclareMathOperator{\tr}{tr}

\DeclareMathOperator{\op}{op}
\newcommand{\core}{\iota_0}

\newcommand{\Mod}{\mathsf{Mod}}

\newcommand{\Vect}{\mathsf{Vect}}
\newcommand{\Hilb}{\mathsf{Hilb}}

\newcommand{\Cat}{\mathsf{Cat}}
\newcommand{\Alg}{\mathsf{Alg}}
\newcommand{\ol}{\overline}

\newcommand{\Z}{\mathbb{Z}}
\newcommand{\R}{\mathbb{R}}
\newcommand{\C}{\mathbb{C}}
\newcommand{\Va}{\mathcal{V}}
\newcommand{\Ca}{\mathcal{C}}

\newcolumntype{C}[1]{>{\centering\arraybackslash}p{#1}}

\def\semicolon{;}
\def\applytolist#1{
    \expandafter\def\csname multi#1\endcsname##1{
        \def\multiack{##1}\ifx\multiack\semicolon
            \def\next{\relax}
        \else
            \csname #1\endcsname{##1}
            \def\next{\csname multi#1\endcsname}
        \fi
        \next}
    \csname multi#1\endcsname}

\def\calc#1{\expandafter\def\csname c#1\endcsname{{\mathcal #1}}}
\applytolist{calc}QWERTYUIOPLKJHGFDSAZXCVBNM;
\def\bbc#1{\expandafter\def\csname bb#1\endcsname{{\mathbb #1}}}
\applytolist{bbc}QWERTYUIOPLKJHGFDSAZXCVBNM;
\def\bfc#1{\expandafter\def\csname bf#1\endcsname{{\mathbf #1}}}
\applytolist{bfc}QWERTYUIOPLKJHGFDSAZXCVBNM;
\def\sfc#1{\expandafter\def\csname s#1\endcsname{{\sf #1}}}
\applytolist{sfc}QWERTYUIOPLKJHGFDSAZXCVBNM;
\def\fc#1{\expandafter\def\csname f#1\endcsname{{\mathfrak #1}}}
\applytolist{fc}QWERTYUIOPLKJHGFDSAZXCVBNM;
\def\rmc#1{\expandafter\def\csname rm#1\endcsname{{\mathrm #1}}}
\applytolist{rmc}QWERTYUIOPLKJHGFDSAZXCVBNM;

\numberwithin{equation}{section}

\theoremstyle{plain}
\newtheorem{thm}[equation]{Theorem}
\newtheorem*{thm*}{Theorem}
\newtheorem{cor}[equation]{Corollary}
\newtheorem{lem}[equation]{Lemma}

\newtheorem{conj}[equation]{Conjecture}
\newtheorem*{claim*}{Claim}

\theoremstyle{definition}
\newtheorem{defn}[equation]{Definition}

\newtheorem*{trick*}{Trick}

\newtheorem{ex}[equation]{Example}

\newtheorem{rem}[equation]{Remark}
\newtheorem*{rem*}{Remark}

\title{The many faces of higher Hilbert spaces
}
\author{Giovanni Ferrer, Lukas M\"uller, David Penneys, and Luuk Stehouwer}
\date{\today}

\begin{document}

\begin{abstract}
Finite-dimensional operator algebras can be viewed as $\rmC^*$, $\rm\rmW^*$, or $\rm\rmH^*$-algebras, leading to different notions for their categories of modules and correspondence 2-categories.
In this article, we show how these differences can be understood systematically using the notion of $G$-dagger category from [\arxiv{2403.01651}] for different subgroups $G\leq O(2)$.
To do so, we first introduce $G$-Hermitian $2$-vector spaces using fixed points of a certain $O(2)$-action on $2\Vect$.
We then propose criteria for when such pairings are `positive', generalizing the passage from Hermitian vector spaces to Hilbert spaces.
Finally, we outline an inductive approach to defining higher Hilbert spaces in arbitrary dimension, suggesting an extension of these ideas beyond the 2-categorical setting.
\end{abstract}

\maketitle


\section{Introduction}
\label{sec:Intoduction}

Categorifications of the notion of Hilbert space are important in quantum information theory \cite{MR3971584,MR2767048} and unitary quantum field theory \cite{MR165864,MR1355899,MR3747830,MR3773743}, especially topological phases of matter \cite{1905.09566,MR4203166,MR4444089,2410.05120}.
As is often the case with categorification, there is not a unique generalization to
higher categories, but many different notions.  
In this note, we discuss a systematic framework organizing and comparing definitions of $2$-categorical unitarity in the finite-dimensional case. 
For us, a {\it 2-vector space} will mean a finite semisimple {(additive idempotent complete)} $\C$-linear category with finitely many isomorphism classes of simple objects. 
We denote the 2-category of 2-vector spaces, linear functors, and natural transformations by $2\Vect$. 
For a 2-vector space $\Va\in 2\Vect$ the following notions have previously been considered:
\begin{itemize}
    \item We can equip $\Va$ with a non-degenerate sequi-linear `inner product' functor $\langle -,- \rangle \colon \overline{\Va} \boxtimes \Va \to \Vect$, directly categorifying the inner product on a Hilbert space. 
    We denote by $2\Vect_{ip}$ the 2-category of 2-vector spaces equipped with an inner-product,linear functors\footnote{ As for linear maps between ordinary Hilbert space, we can additionally consider those invertible functors which preserve the inner product. We will come back to this in the body of this paper.}, and natural transformations. 
    
    \item We can equip $\Va$ with a complex antilinear $\dagger$-structure, i.e. a functor $\dagger \colon \overline{\Va}\to \Va^{\op}$ which is the identity on objects and satisfies $\dagger^2=\id_{\Va}$. 
    A dagger structure on $\Va$ makes all endomorphism algebras into $*$-algebras. 
    The pair $(\Va,\dagger)$ is a \emph{$\rmC^*$-2-vector space} if all of its endomorphism algebras are $\rmC^*$-algebras.\footnote{\label{Footnote:BeingC*Property} Being a $\rmC^*/\rmW^*$-algebra is a property of a complex $*$-algebra and not extra structure; see \cite[\S2.1]{MR3687214}.} 
    We denote by $\rmC^*2\Vect$ the 2-category of $\rmC^*$-2-vector spaces, linear dagger functors, and natural transformations.  

    \item Combining the previous two points we can ask a $\rmC^*$-2-vector space to be equipped with a non-degenerate $\Hilb$-valued inner product, i.e.\ a non-degenerate dagger functor $\langle -,- \rangle \colon \overline{\Va} \boxtimes \Va \to \Hilb$. This structure is naturally present on $\rmW^*$-categories~\cite{MR808930,MR2325696,2411.01678} and for this reason we denote the 2-category of those by $\rmW^*2\Vect$. 

    \item Alternatively, we could also equip $\Va$ with a dagger structure and its hom-spaces with a compatible Hilbert space structure. This structure was introduced by Baez in~\cite{MR1448713} and is often referred to as a \emph{Baez 2-Hilbert space}. We denote the 2-category of Baez 2-Hilbert spaces by $2\Hilb$.     

    \item We can specify compatible non-degenerate symmetric pairings $\kappa_{a,b}\colon \Va(a,b)\otimes \Va(b,a)\to \C$. This structure is known as a Calabi-Yau structure on $\Va$~\cite{MR2298823} and we denote by $\mathsf{CY}2\Vect$ the 2-category of 2-vector spaces equipped with a Calabi-Yau structure. 
    
\end{itemize}

We suggest that the theory of higher dagger categories helps to organize and compare these different notions. For every subgroup $G$ of $O(2)$ there is a notion of $G$-dagger 2-categories~\cite{2403.01651,2505.04761,2504.17764} that are built by picking a collection of \emph{dagger} $G$-Hermitian structures among all $G$-Hermitian structures as we will explain in more detail momentarily. The diagram 
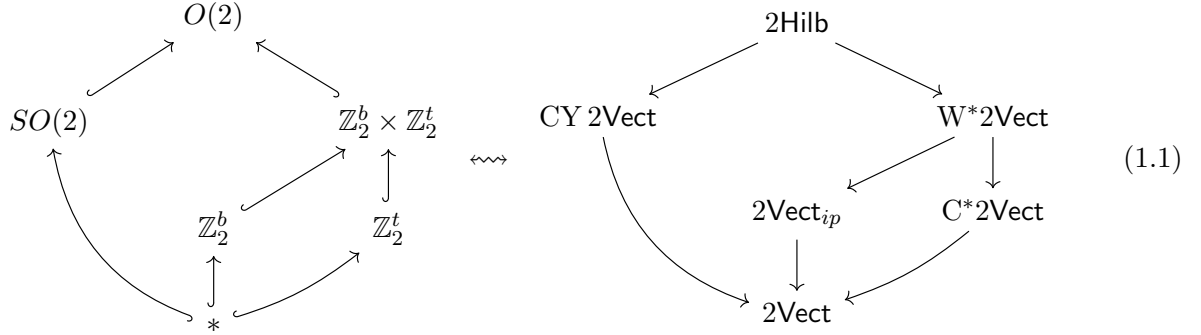
\begin{equation}
\label{salesmanship}
    \begin{tikzcd}
        & O(2) & \\ 
        SO(2) \ar[ru,hookrightarrow] & & \Z_2^b\times \Z_2^t \ar[lu,hookrightarrow] \\ 
        & \Z_2^b \ar[ru,hookrightarrow] & \Z_2^t \ar[u,hookrightarrow] \\ 
        & \ar[luu,hookrightarrow, bend left]\ar[u,hookrightarrow] \ar[ru,hookrightarrow, bend right=10]* & 
    \end{tikzcd} 
    \ \leftrightsquigarrow \ 
    \begin{tikzcd}
        & 2 \Hilb \ar[ld] \ar[rd] & \\ 
        \operatorname{CY}2\Vect \ar[rdd, bend right]  & & \rmW^*2\Vect \ar[d] \ar[ld] \\ 
        & 2\Vect_{ip} \ar[d]  &\rmC^*2\Vect \ar[ld, bend left=10]   \\ 
        & 2\Vect & 
    \end{tikzcd}
\end{equation}
explains the connection between subgroups of $O(2)$ and notions of 2-Hilbert spaces;
this is the main content of this note. 
Here, $\Z_2^b$ and $\Z_2^t$ are the two reflections of the coordinate axis which will play different roles; $\Z_2^b$ will reverse the direction of 1-morphisms (`$b=$ bottom') whereas $\Z_2^t$ reverses the direction of 2-morphisms (`$t=$ top').  
 Following an arrow corresponds to forgetting some of the structure, e.g. $2$-Hilbert spaces have a canonical Calabi-Yau structure, but not every Calabi-Yau category lifts to a $2$-Hilbert space.

To explain our results in more detail, it is helpful to first recall our approach to finite-dimensional Hilbert spaces. 
The category of finite-dimensional vector spaces and linear maps is equipped with an anti-involution
\begin{align}
    \overline{(-)^\vee} \colon \Vect \to \overline{\Vect}^{\op} 
\end{align}
sending a vector space $V$ to its complex conjugated dual $\overline{V^\vee}$. 
This is a twisted action of $O(1)=\Z_2$ on $\Vect$ twisted by the $\Z_2$-action $\Ca \longmapsto \overline{\Ca}^{\op}$, i.e.\, a (homotopy) fixed point for the action $\Ca \mapsto \overline{\Ca}^{\op}$ on the $(2,1)$-category $\Cat_\C$ of $\C$-linear categories, $\C$-linear functors and natural isomorphisms.
There is no notion of fixed points for a twisted action. 
However, suppose we restrict the action to the core $\iota_0 \Vect$, and in addition forget about the $\C$-linear structure.
We can then use the canonical equivalence $\iota_0 \Vect \cong \iota_0 \Vect^{\op} = \iota_0 \overline{\Vect}^{\op}$ which sends every morphism to its inverse to obtain an honest $\Z_2$-action on $\iota_0 \Vect$ in the $(2,1)$-category of groupoids.
A fixed point for this action is a vector space $V$ equipped with an isomorphism $h\colon V\to \overline{V^\vee} $ that satisfies a condition equivalent to $\langle v,w \rangle := h(w)(v)$ defining a Hermitian non-degenerate sesquilinear form on $V$, see Lemma \ref{lem:hermvect}. 

The notion of a finite-dimensional Hilbert space can be defined as such a Hermitian structure which is `positive'; it satisfies $\langle v , v \rangle \geq 0 $ for all $v\in V$. 
We can conclude that the category of Hilbert spaces can be constructed from the category of vector spaces by selecting those $\Z_2$-fixed points corresponding to positive definite inner products. 
This type of approach generalizes to higher linear algebra, as we now sketch and explain in detail in the rest of the paper. 

This story categorifies to $2\Vect$, replacing $O(1)=\Z_2$ by $O(2)$. 
There is a twisted action of $O(2)$ on the 2-category $2\Vect$ of 2-vector spaces, which we call an \emph{$O(2)$-volution}.
The twisted action is defined by combining the twisted $O(2)$-action given by (higher) duals with the $O(1)$-action given by complex conjugation. 
In particular, this leads to twisted actions of subgroups $G\leq O(2)$ on $2\Vect$.
Similar to the case of vector spaces, this induces an honest (that is, non-twisted) action of $O(2)$ on $\iota_0 2\Vect $.  
Since now we are in a $2$-category, we can consider $O(2)$-fixed point structure on objects, as well as 1- and 2-isomorphisms, see Section \S\ref{sec:nHerm} for details. 

The subgroup $\Z_2^t \leq O(2)$ also acts on the category $\iota_1 2 \Vect \supseteq \iota_0 2\Vect$ which only forgets all non-invertible 2-morphisms.
We can thus still talk about $\Z_2^t$-fixed noninvertible $1$-morphisms in $(\iota_1 2 \Vect)^{\Z_2^t}$, even though we can only talk about $O(2)$-fixed point structures for isomorphisms.
Corresponding notions of 2-categories of 2-Hilbert spaces can be constructed by selecting appropriate classes of fixed points on objects and 1-morphisms which we call the class of \emph{$G$-dagger objects}. 
While we introduce our choices of classes of $G$-dagger objects in an ad hoc fashion in order to agree with established classes of operator categories and algebras, we propose an inductive criterion for selecting dagger objects for higher vector spaces in Section~\ref{sec:Inductive}.
 
Note that in the $1$-categorical situation there was a distinction between $O(1)$-fixed and $O(1)$-dagger objects (Hermitian vector spaces versus Hilbert spaces), but no distinction between $O(1)$-fixed morphisms and $O(1)$-dagger morphisms.
  Namely, the latter are \emph{unitary isomorphisms}, meaning $f^\dagger =f^{-1}$, see Lemma \ref{lem:hermvect}.
 We emphasize that for $2$-vector spaces $G$-fixed $1$-morphisms and $G$-dagger $1$-morphisms are allowed to be distinct.
 Being a $G$-fixed $2$-isomorphism is again a condition, which is the same as for $G$-dagger $2$-isomorphisms.
 For $G = O(2)$, the $G$-fixed $2$-isomorphisms are the unitary natural transformations.

\begin{table}[!ht]
    \centering
    \begin{small}
    \begin{tabular}{c||C{35mm} |C{35mm} |C{35mm} |C{35mm}  }
        $G\subset O(2)$ & $G$-fixed object & $G$-dagger object  & $G$-fixed 1-morphism & $G$-dagger 1-morphism  \\ 
        \hline \hline
        * & $2$-vector space & - & - & - \\ 
        \hline
        $\Z_2^b$ & $\Vect$-valued inner product & - & inner product preserving equivalence & - \\
        \hline
        $\Z_2^t$ & anti-involutive category & $\rmC^*$-category & anti-involutive functor & $\dagger$-functor  \\
        \hline
        $\Z_2^b\times \Z_2^t$ & anti-involutive  \& inner product & $\rmW^*$-$2$-vector space  &  anti-involutive \& inner product preserving &  inner product preserving $\dagger$-functor  \\
        \hline
        $SO(2)$ & Calabi-Yau category & - & trace-preserving equivalence & -  \\ 
        \hline
        $O(2)$ & anti-involutive CY-Cat & Baez 2-Hilbert space & anti-involutive \& trace-preserving &  isometric equivalence   \\
    \end{tabular}
\end{small}
    \caption{$G$-fixed objects and $1$-morphisms in $2$-vector spaces corresponding to subgroups of $O(2)$, and which of those are $G$-dagger. If $G$ is not contained in $\Z_2^t$, then $G$-fixed $1$-morphisms are invertible.}
    \label{tab:Cat}
\end{table}

Table~\ref{tab:Cat} summarizes these structures and our choices for dagger objects and $1$-morphisms, which will be explained in more detail below.
We do not know entirely satisfactory definitions of $G$-dagger objects in cases where $G$ does not contain the subgroup $\Z_2^t$. 
The main conceptual reason for this, is that only the 2-morphisms form a vector space over the complex numbers, and in $\C$ we can talk about positive real numbers.\footnote{\label{footnote:HalfBaked}
For $G=\Z_2^b \times \Z_2^t$, there is an under-developed notion of a `positive' 1-morphism \cite{MR4581741}, namely that the 1-morphism should look like ${}_a X\otimes_b X_a^{\dagger_b}$, together with some notion of a `positive collection' of 1-morphisms.
Here $\dagger_b$ is the `bottom dagger', see Table \ref{tab:G dagger}.
In the unitary/$\rmW^*$ setting, this is analogous to looking at a 1-morphism of the form $L^2M$ for some von Neumann algebra $M$ by \cite{MR703809}.
Such 1-morphisms have a canonical $\bbZ_2^b$-fixed point structure.
}


Finally, we explain how the above discussion is related to dagger $2$-categories.
The category $\Hilb$ of finite-dimensional Hilbert spaces and linear maps is equivalent to the category $\Vect$ of finite-dimensional vector spaces as a linear category. What allows us to distinguish these two categories is the existence of a canonical dagger structure on $\Hilb$, defined using the Hermitian adjoint of operators. 

Similarly, we can assemble 
$G$-dagger 2-vector spaces into a $2$-category with additional structure; it is a $G$-dagger 2-category~\cite{2403.01651,2505.04761,2504.17764}. 
As we summarize in Table~\ref{tab:G dagger}, the notion of a $G$-dagger $2$-category unifies several structures on $2$-categories which have been considered before.
Concretely, the $G$-dagger $2$-category constructed from $G$-dagger objects has objects the $G$-dagger objects, $1$-morphisms the $H$-dagger $1$-morphisms where $H \subseteq \Z_2^t$ is the restriction of $G$ to $\Z_2^t$, and $2$-morphisms are arbitrary $2$-morphisms.
For example, the $O(2)$-dagger $2$-category of $O(2)$-dagger $2$-vector spaces has objects Baez $2$-Hilbert spaces, $1$-morphisms dagger functors and $2$-morphisms natural transformations.
We can reinterpret the new layers of $G$-fixed points and $G$-dagger objects as respecting the $G$-dagger category structure.
As this is similar to how $O(1)$-fixed $1$-morphisms correspond to unitary morphisms in a dagger $1$-category, this allows us to reinterpret certain $G$-dagger objects, $1$- and $2$-morphisms in a $G$-dagger $2$-category as \emph{higher unitary morphisms} as we spell out in Table~\ref{tab:G dagger}.
In Section \S\ref{sec:2Hilb}, we will write out this data explicitly for the case of $2$-vector spaces.

\begin{table}[!ht]
    \centering
    \begin{tabular}{c||C{55mm}|C{40mm}|C{40mm}}
        $G\subset O(2)$ &    $G$-dagger structure & unitary 1-isomorphisms & unitary 2-isomorphisms  \\ 
        \hline \hline
    
        $\Z_2^b$  &  $\dagger_b \colon \mathfrak{X} \to \mathfrak{X}^{1\op} $ & $F^{\dagger_b} \cong F^{-1}$ & - \\
        \hline
        $\Z_2^t$  &   $\dagger_t \colon \mathfrak{X} \to \mathfrak{X}^{2\op} $& - & $\eta^{\dagger_t}=\eta^{-1}$ \\
        \hline
        $\Z_2^b\times \Z_2^t$  &   bi-involutive 2-category  & $F^{\dagger_b} \cong F^{-1}$ & $\eta^{\dagger_t}=\eta^{-1}$  \\
        \hline
        $SO(2)$ &  pivotal 2-category & pivotal equivalence &  - \\ 
        \hline
        $O(2)$   &$\dagger_t \colon \mathfrak{X} \to \mathfrak{X}^{2\op} $ + unitary adjoint functor  & isometric equivalence & $\eta^{\dagger_t}=\eta^{-1}$ \\
    \end{tabular}
    \caption{$G$-dagger structures on a $2$-category $\mathfrak{X}$. There is a relationship between $1$- and $2$-morphisms preserving this dagger structure, and preserving $G$-fixed point data.
    More specifically, the $G$-dagger morphisms in Table \ref{tab:Cat} are the unitary morphisms for the case where $\mathfrak{X}$ is the $G$-dagger $2$-category of `$G$-Hilbert spaces'.
    Above, the notion of isometric equivalence is defined in \cite[Def.~4.10]{2410.05120}.}
    \label{tab:G dagger}
\end{table} 

We also translate the above results to 2-categories of finite dimensional algebras in \S\ref{sec:Alg}. 
In doing so, we reproduce fundamental formulas and constructions from the theory of operator algebras just using categorical principles.
For example, in \S\ref{Sec: Alg Z2t}, we obtain the formula for the operator-valued inner product on the relative tensor product of Hilbert $\rmC^*$-correspondences \cite{MR0367670,MR703809,MR1303779} by composing $\bbZ_2^t$-fixed point data.
In \S\ref{Sec: Alg Z2TaB}, we see the notion of the Haagerup standard form of a von Neumann algebra \cite{MR0407615} appear from a $\bbZ_2^b\times \bbZ_2^t$-fixed point structure, selecting $\rmW^*$-algebras together with their canonical standard form $L^2(A)$ as our choice of $\bbZ_2^b\times \bbZ_2^t$-dagger objects. 
We also recover the Connes fusion for the relative tensor product of bimodules over these $\bbZ_2^b\times \bbZ_2^t$-dagger objects.
Even though our results are in the finite dimensional setting, we expect them to generalize to infinite dimensional $\rmC^*$ and $\rmW^*$-algebras.

\subsection{Outlook}

Many of the structures we obtain in Section \S\ref{sec:2dag} generalize from $2$-vector spaces to $\C$-linear categories with infinite-dimensional hom-spaces and/or infinitely many objects in an ad-hoc fashion.
Given the lack of duals and adjoints in such $2$-categories, our definition of $O(2)$-Hermitian vector space no longer makes sense in this situation.
Therefore such generalizations require some care in replacing duals and fixed points by weaker versions, similarly to how for infinite-dimensional vector spaces $V$ there is still a dual vector space $V^\vee$ and a Hilbert space is a certain map $V \hookrightarrow \overline{V}^\vee$.
See \cite{luders2026relative} for some results in this direction.
Describing infinite-dimensional $2$-Hilbert spaces more systematically would be desirable, especially for the application of infinite-dimensional higher Hilbert spaces to non-topological quantum field theory.

Separately, there is a diagram for a general 3-Hilbert space $\fX$ corresponding to Diagram \ref{salesmanship}.
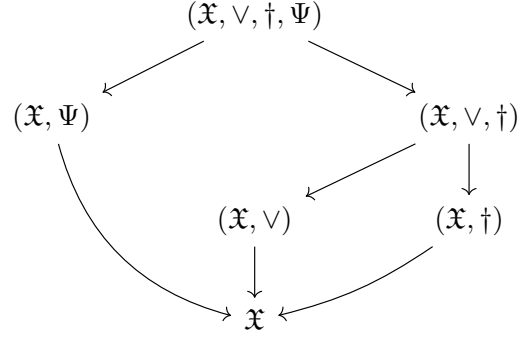
\[
    \begin{tikzcd}
        & (\fX,\vee,\dag,\Psi) \ar[ld] \ar[rd] & \\ 
        (\fX,\Psi) \ar[rdd, bend right]  & & (\fX,\vee,\dag) \ar[d] \ar[ld] \\ 
        & (\fX,\vee) \ar[d]  &(\fX,\dag) \ar[ld, bend left=10]   \\ 
        & \fX & 
    \end{tikzcd}
\]
One could perform a similar analysis to \eqref{salesmanship} internal to any such 3-Hilbert space, e.g., $2\mathsf{sHilb}$.

Finally, it would be interesting to find a correspondence along the lines of \eqref{salesmanship} one categorical level higher for various objects between 3-Hilbert spaces and 3-vector spaces; see Conjecture \ref{conj:O3FixedPoint}.

\subsection*{Acknowledgments}
The authors would like to thank 
Andr\'e Henriques, Theo Johnson-Freyd,
Corey Jones,
and David Reutter
for helpful conversations.
Giovanni Ferrer and David Penneys were supported by NSF DMS-2154389, and David Penneys was also supported by NSF DMS-2554723.
Lukas M\"uller and David Penneys would also like to thank CMSA for hospitality during the twinned workshop on Quantum Field Theory and Topological Phases via Homotopy Theory and Operator Algebras.
Luuk Stehouwer was supported by the ERC Consolidator Grant ``SYMSPEC".

\section{Higher Hermitian vector spaces using \texorpdfstring{$O(n)$}{O(n)}-fixed points} 
In this section we discuss higher analogues of Hermitian vector spaces for categories of higher vector spaces. We start with some generalities, but mostly focus on the case $n=2$. 
\subsection{\texorpdfstring{$n$}{n}-Herm}
\label{sec:nHerm}

Let $n\Vect$ be the symmetric monoidal $n$-category of fully dualizable $n$-vector spaces.
Here $n$-vector spaces are defined inductively by condensation completion as in \cite{1905.09566}.
The category $1\Vect$ is the ordinary category of finite-dimensional vector spaces. 

The 2-category $2\Vect$ can be identified with the 2-category of finite semisimple $\C$-linear categories, 
linear functors, and natural transformations. An equivalent description is in terms of the Morita $2$-category $\Alg$ of semisimple finite-dimensional algebras, finite-dimensional bimodules, and bimodule maps. Indeed, a special case of the Eilenberg-Watts theorem implies that the 2-functor $\operatorname{Mod}\colon \Alg \rightarrow 2\Vect$ sending a finite semi-simple algebra to its category $\operatorname{Mod}_A$ of finite-dimensional modules is an equivalence of $2$-categories. 
From now on, we will describe $2$-vector spaces as $\C$-linear categories and refer the reader interested in comparing with the $2$-category of algebras to Section \S\ref{sec:Alg}.

The 3-category $3\Vect$ can be identified with semisimple $\C$-linear 2-categories with finitely many isomorphism classes of simple objects or equivalently the Morita 3-category of multifusion categories~\cite{1812.11933,MR4372801}. 

The $n$-category $n\Vect$ is $n$-rigid\footnote{I.e.\ all its objects are dualizable and less than $n$-morphisms have both a left and right adjoint.}, and hence
by the cobordism hypothesis~\cite{MR2555928}, there is an $O(n)$-action on $\core(n\Vect)$, the maximal subgroupoid of $n\Vect$, generalizing the $O(1)$-action on $\core(\Vect)$ that sends a vector space to its dual. 
In addition, the $n$-category $n\Vect$ has a symmetric monoidal ($\C$-antilinear) $O(1)$-action given by complex conjugation which, as any symmetric monoidal action, commutes with the $O(n)$ action from the cobordism hypothesis. 
Explicitly, the universal property of Cauchy-completion allows us to inductively lift the $O(1)$-action of conjugation $\overline{\,\cdot\,}$ on $\bbC$ as follows:
\[
\begin{tikzcd}
n\Vect \ar[r,dashed,"\overline{\,\cdot\,}"] & n\Vect \\
\rmB (n-1)\Vect \ar[u,hookrightarrow] \ar[r,dashed,"\overline{\,\cdot\,}"] & \rmB(n-1)\Vect \ar[u,hookrightarrow] \\ 
\vdots \ar[u,hookrightarrow] & \vdots \ar[u,hookrightarrow]\\
\rmB^{n-1} \Vect \ar[r,dashed,"\overline{\,\cdot\,}"] \ar[u,hookrightarrow] & \rmB^{n-1} \Vect \ar[u,hookrightarrow]\\
\rmB^n \bbC \ar[r,"\overline{\,\cdot\,}"] \ar[u,hookrightarrow] & \rmB^n \bbC \ar[u,hookrightarrow] 
\end{tikzcd}
\]
From the construction it is clear that the higher cells needed for the $O(1)$-action on $n\Vect$ lift to identities and that the action is indeed symmetric monoidal. 
By combining complex conjugation and the cobordism hypothesis action we get a new $O(n)$-action $\alpha_n$ by composing with the diagonal
$$
O(n) \xrightarrow{(\id, \det)} O(n) \times O(1).
$$
Note that $\alpha_1$ is the $O(1)$-action $\ol{(-)}^\vee$ on $\core(\Vect)$.
In~\cite{2403.01651}, we argue that this $O(n)$-action is the restriction of what we suggest to call an $O(n)$-\emph{volution}---a certain twisted action of $O(n)$---on the full $n$-category $n\Vect$. In~\cite{2dunitarity} it will be shown in detail that this expectation holds for $n=2$. We briefly summarize the findings in the following remark.

\begin{rem} 
The topological group $O(2)=SO(2)\rtimes \Z_2^{t}$ is homotopy equivalent to the 2-group $\rmB\Z \rtimes \Z_2^{t}$. In~\cite{2dunitarity}
(see also~\cite[Ex.~5.5]{2403.01651}), we will show that this $2$-group acts on the $(3,1)$-category $\mathsf{AdjCat}_2$ of 2-categories with adjoints. 
The $\Z_2^{t}$-action is given by sending a 2-category $\mathfrak{X}$ to the 2-category $\mathfrak{X}^{2 \op}$ and the generating 1-morphism of $\rmB\Z$ acts by the natural isomorphism $(-)^{LL}\colon \id_{\mathsf{AdjCat}_2} \to \id_{\mathsf{AdjCat}_2 }$ whose component at a 2-category $\mathfrak{X}$ is the functor $(-)^{LL}_{\mathfrak{X}}\colon \mathfrak{X}\longrightarrow \mathfrak{X}$ that is the identity on objects and sends 1-morphisms to their double left adjoint.    

An \emph{$O(2)$-volutive 2-category} is a fixed point for this action. 
Every fully dualizable symmetric monoidal $2$-category has a canonical $O(2)$-volution which extends the cobordism hypothesis action on the core to the whole 2-category. 
In~\cite{2dunitarity} (see also~\cite[Stat.~6.1]{2403.01651}), 
we prove that symmetric monoidal $O(2)$-volutions on a fully dualizable symmetric monoidal $2$-category correspond to symmetric monoidal $O(2)$-actions on that $2$-category.
The equivalence is given by twisting the canonical $O(2)$-volution by the given symmetric monoidal $O(2)$-action.

The resulting $O(2)$-volution on $\mathfrak{X} = 2\Vect$ when we twist the canonical $O(2)$-volution by complex conjugation can be summarized as follows (see~\cite{2dunitarity} for all the coherences)
\begin{itemize}
    \item The $\Z_2^t$-part acts by $2\Vect \longrightarrow 2\Vect^{2\op}, \ \ \Va \longmapsto \overline{\Va}^{\op}$
    \item The $\Z_2^b$ acts by $2\Vect \longrightarrow 2\Vect^{1\op}, \ \Va \mapsto \Hom(\Va, \Vect)$.
    \item The generating 1-morphism of $\rmB\Z$ acts by a natural isomorphism $S \colon \id_{2\Vect}\Longrightarrow (-)^{LL}_{2\Vect}$ whose component at a 2-vector space $\Va$ is most conveniently described as a profunctor
    \begin{align}
   S_{\Va} \colon  \Va^{\op} \boxtimes \Va & \longrightarrow \Vect \\
   a\boxtimes b & \longmapsto \Va(b,a)^*.
    \end{align}
    Here given $x, y \in \Va$, the enriched hom space is denoted $\Va(x,y) \in \Vect$.
\end{itemize}
\end{rem}
{
\begin{rem}
The reader might wonder why we only twisted the cobordism hypothesis $O(n)$-action by an $O(1)$-action as opposed to a full $O(n)$-action.
Our bias to do so is the same bias that makes differential topologists work with oriented manifolds.
When replacing vector spaces with supervector spaces and hence oriented manifolds with spin manifolds, the $O(n)$-action we would twist by would no longer factor through $O(1)$.
At the other extreme, when working with real vector spaces and hence unoriented manifolds, we do not need to twist at all.
\end{rem}}

To motivate our definition of a \emph{Hermitian} $n$-vector space, we recall from the introduction the observation that for $n=1$,
a fixed point for the action $\alpha_1 = \ol{(-)}^\vee$ is exactly a nondegenerate Hermitian form on a vector space.
We summarize this in the following lemma.

\begin{lem}
\label{lem:hermvect}
    The groupoid
    \[
    \core(\Vect)^{O(1)}
    \]
    of fixed points for the $O(1)$-action $\alpha_1 = \ol{(-)}^\vee$ on $\core(\Vect)$ has objects finite-dimensional vector spaces $V$ equipped with a nondegenerate (but not necessarily positive definite) Hermitian form $h_V \colon V \to \ol{V}^\vee$ and morphisms $(V,h_V) \to (W,h_W)$ are unitary maps, i.e. isomorphisms $f \colon V \to W$ such that
    \[
    f^\dagger \colon W \xrightarrow{h_W} \ol{W}^\vee \xrightarrow{\ol{f}^\vee} \ol{V}^\vee \xrightarrow{h_V^{-1}} V
    \]
    is the inverse of $f$.
\end{lem}
This motivates the following definition for $n>1$:
\begin{defn}
\label{def:nHermVect}
A \emph{Hermitian $n$-vector space} is an $O(n)$-fixed point for $\alpha_n$. 
More generally, for a subgroup $G\leq O(n)$, a \emph{$G$-Hermitian $n$-vector space} is a fixed point for the induced $G$-action on $\iota_0(n\Vect)$. 
\end{defn}

\begin{rem}
Note that $(n-1)\Vect \in n\Vect$ carries canonically the structure of an $O(n)$-Hermitian $n$-vector space because the $O(n)$-action on $\core(n\Vect)$ is symmetric monoidal and so preserves the unit.
\end{rem}

Definition \ref{def:nHermVect} provides a whole $n$-groupoid of $G$-Hermitian $n$-vector spaces.
In particular, given a $1$-isomorphism of $n$-vector spaces, the structure of an equivalence of $G$-Hermitian $n$-vector spaces amounts to a lift along the functor $\core(n\Vect)^{G} \to \core(n\Vect)$.
For $n=1$ this is only a condition and recovers the notion of unitary equivalence between Hermitian vector spaces.

More generally, we can {isolate} certain \emph{non-invertible} $1$-morphisms between Hermitian $n$-vector spaces as follows.
Consider the $(n,1)$-category $\iota_1 n\Vect$ forgetting all non-invertible $k$-morphisms for $k >1$. 
A consequence of the cobordism hypothesis is that $O(n-1)$ acts on $\iota_1 n\Vect $. 
We can thus consider the category of fixed points $(\iota_1 n\Vect)^{O(n-1)}$ combining the action from the cobordism hypothesis with complex conjugation.  
 More generally, for a subgroup $G \leq O(n-1)$, we can talk about $G$-fixed point structures on noninvertible $1$-morphisms. 

 \begin{rem}
     Even though our definition of $G$-Hermitian $n$-vector space compiles for any $G \leq O(n)$, we choose in this work to focus on a restricted class of subgroups of $O(2)$.
     More specifically, 
     \begin{itemize}
         \item whenever $G$ contains a rotation by less than $180$ degrees we require $G$ contain all of $SO(2)$, and 
         \item whenever $G$ contains a reflection different from the two standard coordinate reflections we require $G$ to be all of $O(2)$.\footnote{Observe that we could still allow $G$ to be the diagonal $\Z_2$-subgroup consisting of a single rotation by $180$ degrees, but we chose to exclude this case for expository reasons. }
     \end{itemize}  
     The intuitive reason is that the standard string diagram calculus of $2$-categories takes place in a square as opposed to a sphere.
     
     In more detail, given a $2$-category $\mathfrak{X}$, the result of `reflecting string diagrams along the bottom coordinate' is $\mathfrak{X}^{1\op}$, while its `reflection along the top coordinate' is $\mathfrak{X}^{2\op}$.
     Since we do not know how to compose morphisms other than vertically and horizontally, we are not aware how to reflect a $2$-category over any other axis.
     If $1$-morphisms of $\mathfrak{X}$ have adjoints, a further allowed operation is to rotate string diagrams by 180 degrees, resulting in a functor $(-)^L\colon \mathfrak{X}\to \mathfrak{X}^{1\op,2\op}$ (after choosing the direction of rotation so that it in the given conventions it recovers the left adjoint).
     However, we are not able to give an interpretation to rotations by a smaller angle.\footnote{In a double category, rotations with 90 degrees do make sense using conjoints and companions.
     In another direction, morphisms in `disk-like' higher categories \cite{MR2978449} do not have sources and targets, and so other rotations could be analyzed in this context.
     }

     To make the description of the `allowed' subgroups of $O(2)$ more systematic, note that as a $2$-group, $O(2)$ can be represented non-minimally as the semidirect product $(\Z_2^b \times \Z_2^t) \ltimes B\Z$. 
     Here both reflection $\Z_2$'s act nontrivially on the $B\Z$ factor of which the generator is given by the path $\operatorname{rot}_t$ for $t \in [0,\pi]$ connecting the identity with the rotation by $180$ degrees.
     The subgroups we are interested in correspond to simplicial subgroups in this simplicial model of $O(2)$.
     We expect understanding of an analogous `cube-like' simplicial decomposition of the $\infty$-group $O(n)$ to be important for deeper understanding of $n$-Hilbert spaces.
     However, at the moment we are not aware how to make this systematic, and hope to return to this point in future work.
 \end{rem}

\subsection{\texorpdfstring{$2$}{2}-Herm}
\label{sec:2Herm}

In this subsection, we spell out the definition of $G$-Hermitian $2$-vector space explicitly for several choices of $G$.

    \subsubsection{$\Z_2^b$} The action of $\Z_2^b$ of the cobordism hypothesis on the core of a fully dualizable $n$-category is given by the dual functor.
    For $2\Vect$, one model of the dual of a 2-vector space $\Va$ is $\mathcal{V}^{\vee}\coloneqq \Hom(\mathcal{V},\Vect)$.
    Combining this $\Z_2^b$-action with complex conjugation gives the $\Z_2^b$-action $\mathcal{V} \mapsto \ol{\mathcal{V}}^\vee$.
Hence, a $\Z_2^b$-Hermitian $2$-vector space consists of a linear category $\mathcal{V}$ and an equivalence $D\colon \mathcal{V} \to \overline{\mathcal{V}}^{\vee}$ together with a natural isomorphism $\theta\colon \overline{D}^{\vee, -1} \circ D \Longrightarrow \id_{\Va}$ satisfying a natural coherence condition, where we implicitly use the identification $\overline{(-)}^{\vee} \circ \overline{(-)}^{\vee} \cong \id_{2\Vect}$. 
Such $\Z_2^b$-fixed point data are equivalent to non-degenerate $\Vect$-valued inner products 
\begin{align}
  \langle -,-\rangle \colon  \overline{\Va} \boxtimes \Va \xrightarrow{D \boxtimes \id_{\Va}} \Va^{\vee} \boxtimes \Va \xrightarrow{\operatorname{ev}} \Vect
\end{align} which are symmetric in the sense that $\theta$ corresponds to coherent isomorphisms $\langle a, b \rangle \cong \overline{\langle b , a \rangle }$ such that $\theta^2 = \id$, compare \cite[\S 3]{MR1256993} and \cite{2411.01678}. 

There is no notion of $\Z_2^b$-fixed point on a general 1-morphism because the $\Z_2^b$-action reverses sources and targets. 
However, for invertible 1-morphisms $F \colon \mathcal{V} \to \mathcal{W}$, we can require the extra datum of $F$ commuting with $D$, or equivalently $F$ preserving the categorified inner product.
Explicitly this compatibility datum consists of natural isomorphisms
\begin{equation}
\label{eq:bottomunitary}
\langle F(x), y \rangle_\cW \cong \langle x, F^{-1}(y) \rangle_\cV
\end{equation}
between functors $\overline{\mathcal{V}} \boxtimes \mathcal{W} \to \Vect$
satisfying a condition with respect to the isomorphisms 
$\langle x_1, x_2 \rangle_\cV \cong \ol{\langle x_2,x_1 \rangle_\cV}$
and
$\langle y_1, y_2 \rangle_\cW \cong \ol{\langle y_2,y_1 \rangle_\cW}$.
In this sense, the fixed point data on $F$ can be seen as a categorification of the notion of a unitary morphism, analogously to Lemma \ref{lem:hermvect}.
A natural isomorphism between inner product preserving linear functors is a $\Z_2^b$-fixed point exactly when it intertwines the respective isomorphisms \eqref{eq:bottomunitary}.

\begin{ex}
\label{ex:vect}
    A short computation shows that there is only one $\Vect$-valued inner product on $\Vect$ up to inner product preserving equivalences.
    This $\Z_2^b$-fixed point is given by $\langle V,W \rangle = \overline{V} \otimes W$ equipped with the obvious flipping isomorphism $\langle V,W \rangle \cong \overline{\langle W,V \rangle}$.
\end{ex}

\subsubsection{$\Z_2^t$} 
The action of the subgroup $\Z_2^t\leq O(2)$ sends a 2-vector space $\Va$ to $\overline{\Va}^{\op}$. 
The corresponding fixed points are 2-vector spaces $\Va$ equipped with an anti-involution $(d: \mathcal{V} \to  \overline{\mathcal{V}}^{\op}, \eta: \id_\mathcal{V} \Rightarrow d^2)$. 
A $\Z_2^t$-fixed $1$-morphism between $\Z_2^t$-Hermitian $2$-vector spaces is an anti-involutive functor and a $\Z_2^t$-fixed $2$-morphism is an anti-involutive natural transformations, see \cite[Section 2]{MR4845668} for detailed definitions.
Observe that in contrast to $\Z_2^b$, a $\Z_2^t$-fixed point structure on a functor $F$ does not require it to be invertible.  


Note that $\Z_2^t$-Hermitian $2$-vector spaces can be identified with $\Z_2^b$-Hermitian $2$-vector spaces; the two subgroups $\Z_2^t$ and $\Z_2^b$ are homotopic inside $O(2)$ and hence define equivalent actions on $\iota_0 2\Vect$.
On objects this can be seen by noting that the hom-pairing $\Va^{\op}\boxtimes \Va\to \Vect$ realizes $\Va^{\op}$ as the dual of $\Va$, giving an equivalence $\Va^{\op} \cong \Va^\vee$. 
Explicitly the resulting correspondence between anti-involutions and $\Vect$-valued inner products is given by
\begin{align}
\label{eq:innprodishom}
    \langle x , y \rangle \cong \Va(dx,y)
\end{align}
equipped with the natural isomorphism
\begin{equation}
\label{eq:cathermitian}
\langle x , y \rangle = \Va(dx,y) \xrightarrow{d} \ol{\Va(dy, d^2 x)} \xrightarrow{\eta} \ol{\Va(dy, x)} = \ol{\langle y , x \rangle}.
\end{equation}
We cannot compare the two $\Z_2$-actions on all of $2 \Vect$ however, since the $2$-functor $\Va \mapsto \Va^{\op}$ naturally reverses the direction of $2$-morphisms, while $\Va \mapsto \Va^{\vee}$ reverses $1$-morphisms.

\subsubsection{$\Z_2^b \times \Z_2^t$} 
A $\Z_2^b\times \Z_2^t$-Hermitian $2$-vector space consists of $\Z_2^b$- and $\Z_2^t$-fixed point data on a $2$-vector space $\mathcal{V}$ with a compatibility 2-isomorphism  
\begin{equation}
    \begin{tikzcd}
        \Va \ar[r,"d"] \ar[d,"D",swap] & \overline{\Va}^{\op} \ar[d,"\overline{D}^{\op}"] \ar[ld, Rightarrow] \\ 
        \overline{\Va}^{\vee} \ar[r,"\overline{d}^{\vee, -1}",swap] & \Va^{\vee, \op} 
     \end{tikzcd}
\end{equation} 
satisfying several conditions.
For this diagram to compile we implicitly used the identifications 
$$
(\Va^\vee)^{\op} 
= 
\Hom(\mathcal{V}^{\op}, \Vect^{\op}) 
\cong
\Hom(\mathcal{V}^{\op}, \Vect) 
=
(\mathcal{V}^{\op})^\vee,
$$ 
and 
$$
\ol{\Va^\vee}
= 
\Hom(\ol{\mathcal{V}}, \ol{\Vect}) 
\cong
\Hom(\ol{\mathcal{V}}, \Vect) 
=
(\ol{\mathcal{V}})^\vee,
$$ 
where we used the equivalences $\overline{(-)} \colon \Vect \to \overline{\Vect}$ and  $(-)^\vee \colon \Vect^{\op}\to \Vect$ respectively.
Translating to the level of inner products results in $\ol{\langle d(-), - \rangle}^\vee \cong \langle -, d^{-1}(-) \rangle$.  
We thus get natural isomorphisms 
\begin{equation}
\label{eq:aiinnerproduct}
\langle dx, dy \rangle \cong \overline{\langle x,y \rangle}^\vee
\end{equation}
saying that the inner product defines an anti-involutive functor $\ol{\Va} \boxtimes \Va \to \Vect$.

For 1-isomorphisms between $\Z_2^b \times \Z_2^t$-fixed $2$-vector spaces, a $\Z_2^b\times \Z_2^t$-fixed point is an anti-involutive structure equipped with compatible isomorphisms $\langle F(x), F(y) \rangle \cong \langle x, y \rangle$. 
A natural isomorphism preserves $\Z_2^b \times \Z_2^t$-fixed point data simply when it preserves both the $\Z_2^b$- and $\Z_2^t$-fixed point data.

\subsubsection{$SO(2)$} 
An $SO(2)$-Hermitian 2-vector space consists of a $2$-vector space $\Va$ equipped with a trivialization of the profunctor 
\begin{align}
   S_\Va \colon \Va^{\op}\boxtimes \Va & \longrightarrow \Vect \\ 
    a\boxtimes b &\longmapsto \Hom(b,a)^\vee.
\end{align}
This data is equivalent to a Calabi-Yau structure on $\Va$, i.e., a collection of non-degenerate traces $\tau_c: \End(c) \to \bbC$ for all $c \in \Va$. 

An $SO(2)$-Hermitian 1-isomorphism is an equivalence which preserves the trace. This is equivalent to the notion of a pivotal equivalence from~\cite[Defn.~2.2 and Lem.~5.8]{2307.06485}. 

\subsubsection{$O(2)$} 
Putting everything together, an $O(2)$-Hermitian $2$-vector space consists of a $2$-vector space $\mathcal{V}$, an anti-involution $(d: \mathcal{V} \to  \mathcal{V}^{\op}, \eta: \id_\mathcal{V} \Rightarrow d^2)$ and a Calabi-Yau structure $\tau_c: \End(c) \to \bbC$ for $c \in \mathcal{V}$ which are compatible in the sense that $\kappa(f,g)=\overline{\kappa(d(f),d(g))}$ where $\kappa$ is the bilinear pairing on hom spaces induced by $\tau$. 

The underlying $\Z_2^b$-Hermitian $2$-vector space is given by composing $d$ with the equivalence $\mathcal{V}^{\op} \cong \mathcal{V}^\vee$ given by the Yoneda embedding as explained above. 
The underlying $\Z_2^b \times \Z_2^t$-Hermitian $2$-vector space involves the isomorphism \eqref{eq:aiinnerproduct} which in this case is the isomorphism
\begin{equation}
\label{eq:Z2xZ2}
\ol{\langle dx, y \rangle}^\vee \overset{\eqref{eq:cathermitian}}{\cong} 
\langle y, dx \rangle^\vee = 
\mathcal{V}(y,dx)^\vee \xrightarrow{\kappa} \mathcal{V}(dx,y) = \langle dx, y \rangle.
\end{equation}

\section{Organizing notions of 2-Hilbert spaces into \texorpdfstring{$G$}{G}-dagger 2-categories}
\label{sec:2dag}
We briefly recall the 1-categorical situation. 
Structures on a $1$-category involving unitarity are conveniently encoded by a dagger structure, i.e.\ a functor $\dagger \colon \Ca \to \Ca^{\op}$ such that $\dagger^2 = \id_{\Ca}$ which is the identity on objects. 
For the generalization to higher categories it is convenient to reformulate this in terms of an anti-involutive structure together with the choice of $O(1)$-fixed point data on objects. 
Concretely, the main result of~\cite{MR4845668} is that dagger categories are equivalent to anti-involutive categories $\Ca$ together with a full subgroupoid $\Ca_0 \subset (\iota_0 \Ca)^{O(1)}$, such that $\Ca_0\to \Ca$ is essentially surjective. 
We have seen that in the example $\Ca = \Vect$ equipped with the anti-involution $V \longmapsto \overline{V}^\vee$, the groupoid $(\iota_0 \Ca)^{O(1)}$ consists of nondegenerate Hermitian forms.
The dagger category of Hilbert spaces corresponds to the full subgroupoid on the Hilbert spaces. 

This perspective on dagger categories underlies the generalization to higher categories proposed in~\cite{2403.01651}; specifying a $G$-dagger structure on a $G$-volutive 2-category $\mathfrak{X}$ is the same as selecting a collection of $G$-fixed point data for objects and 1-morphisms. Interestingly in the higher categorical world the Hermitian structures are not a \emph{full} subgroupoid of $(\iota_0 \mathfrak{X})^G$ allowing for a refinement of one $G$-Hermitian object into many inequivalent chosen structures. We will discuss common choices for these in the next section, which will illuminate this point.  

{New version of above paragraph: This perspective on dagger categories underlies the generalization to higher categories proposed in~\cite{2403.01651}; specifying a $G$-dagger structure on a $G$-volutive 2-category $\mathfrak{X}$ involves selecting a collection of $G$-fixed point data for objects and 1-morphisms. 
However, in the higher categorical world the $G$-dagger structure will no longer be fixed by a full subgroupoid of $(\iota_0 \mathfrak{X})^G$.
Instead, we will also need to choose certain fixed point data on non-invertible $1$-morphisms using the $O(1)$-action on $\iota_1 \mathfrak{X}$.
We will discuss common choices for these fixed points in the next section, which will illuminate this point.}
\subsection{Dagger objects in \texorpdfstring{$2\Vect$}{2Vect}}
\label{sec:2Hilb}

The upshot of the previous discussion is that to construct a dagger $2$-category, and hence talk about higher categorical unitarity, we choose fixed points. 
For every subgroup $G\leq O(2)$, we could thus attempt to make a $G$-dagger $2$-category of `$G$-$2$-Hilbert spaces' by selecting those $G$-fixed $2$-vector spaces (and selecting fixed $1$-morphisms between those) which `look positive'.
We will explain how for certain choices of $G$-dagger objects
we do recover well-studied notions in the literature as special cases of $G$-fixed points, such as $\rmC^*$-categories for $\Z_2^t$ and $2$-Hilbert spaces for $O(2)$. 
At this moment, we are not aware of a systematic way to do this for general $G$, but in Section \S\ref{sec:Inductive} we do explain a conceptual way how to define $G$-dagger objects for $G$ containing $\Z_2^b \times \Z_2^t$.

\subsubsection{$\Z_2^b$:}
For $G=\Z_2^b$ we do not know a good notion of preferred Hermitian structures (see Footnote \ref{footnote:HalfBaked} however). 
Hence, for now we just choose all of them and denote the resulting category by $2\Vect_{ip}$ for 2-vector spaces equipped with a vector space valued inner product. 

Motivated by Lemma \ref{lem:hermvect}, we define the identity on objects 2-functor 
\begin{align}
\label{eq:bottomdag}
    \dagger_b \colon 2\Vect_{ip} & \longrightarrow  2\Vect_{ip}^{1\op} \nonumber
    \\ 
    (\Va,D,\theta) &\longmapsto  (\Va,D,\theta) \\ \nonumber
    F\colon (\Va,D,\theta) \to (\Va' ,D',\theta') &  \longmapsto F^{\dagger_b} \coloneqq D \circ \overline{F}^\vee \circ  D'^{-1} \ \ .
\end{align}
Analogously to \eqref{eq:bottomunitary}, we have that that $F^{\dagger_b}$ is an `adjoint' with respect to the $\Vect$-valued inner product induced by the fixed point data, i.e.\, $\langle c' , F(c) \rangle_{\Va'} \cong \langle F^{\dagger_b}(c'),c \rangle_{\Va}$.
This gives a `non-unitary' version of \cite[Definition 4.14]{2411.01678} in the finite-dimensional case.
Together with the natural isomorphism $(\dagger_b)^2 \cong \id$ satisfying a condition, we call this structure a $\Z_2^b$-dagger $2$-category.

\subsubsection{$\Z_2^t$:}
The situation is different for $G=\Z_2^t$, where we have seen that a fixed point is the same as an anti-involutive 2-vector space $(\mathcal{V},d)$. 
Among anti-involutive categories, dagger categories play an important role when discussing positivity, and hence the first subclass of $\Z_2^t$-Hermitian 2-vector spaces we restrict to are those corresponding to dagger structures. 
Similarly to the situation discussed above, inequivalent dagger categories $(\mathcal{V},\dagger)$ with the same underlying anti-involutive category $(\mathcal{V},d)$ are parametrized by possible choices of full subgroupoids of $O(1)$-fixed points on its objects. In particular, there are many inequivalent dagger structures on the same anti-involutive category as alluded to in the previous section. Correspondingly, we restrict to those $\Z_2^t$-fixed 1-morphisms which are not just anti-involutive functors, but dagger functors. 
This yields a $\Z_2^t$-dagger 2-category $2\Vect^\dagger$ of 2-vector spaces equipped with a $\C$-antilinear dagger structure, $\C$-linear dagger functors between them, and natural transformations between those.

Similar to the fact that the morphisms in $\Hilb$ are arbitrary linear maps, we took $2$-morphisms of $2\Vect^\dagger$ to be arbitrary natural transformations.
The unitary natural transformations are encoded by the $\Z_2^t$-dagger structure, which manifests itself in terms of the existence of a natural transformation $\varphi^{\dagger_t} \colon G \Rightarrow F$ for every natural transformation $\varphi \colon F \Rightarrow G$ {defined by $(\varphi^{\dagger_t})_x := (\varphi_x)^\dagger$}.\footnote{{We warn the reader not to confuse the \emph{internal dagger} $\dagger$ on the objects of $2\Vect^\dagger$ with the (related) dagger structure $\dagger_t$ on the whole $2$-category.}}  
The operation $(-)^{\dagger_t}$ is a strictly involutive 2-functor 
\begin{align}
(-)^{\dagger_t} \colon  2\Vect^\dagger \longrightarrow (2\Vect^\dagger)^{2\op}
\end{align}
which is the identity on objects and 1-morphisms.
This leads to the notion of unitary 2-morphisms as natural transformations satisfying $\varphi^{-1}=\varphi^{\dagger_t}$, and positive natural transformations as those of the form $\psi^{\dagger_t} \psi$.
What we call a $\Z_2^t$-dagger $2$-category is usually just called `dagger $2$-category' in the literature  (see e.g.\ \cite[Definition 2.5]{MR3584697}, \cite{MR4326107} and \cite{MR4419534}), while our notion of $G$-dagger $2$-category encompasses more concepts.

Even though we could talk about `positive $2$-morphisms' between $2$-vector spaces equipped with a dagger structure, they aren't `positive' in the sense of `positive real numbers'. 
In comparison, elements of the form $a^*a$ in a $*$-algebra $A$ only necessarily have positive spectrum when $A$ is a $\rmC^*$-algebra.
To relate this observation to $2$-vector spaces, note that a complex antilinear dagger structure on a $2$-vector space $\mathcal{V}$ makes all endomorphism algebras into $*$-algebras. 
Then $\mathcal{V}$ is called a \emph{$\rmC^*$-2-vector space} (i.e.\ a $\rmC^*$-category that is additionally a $2$-vector space) if all endomorphism algebras are $\rmC^*$ (see Footnote \ref{Footnote:BeingC*Property}).
Imposing this additional condition leads to the $\Z_2^t$-dagger 2-category $\rmC^*2\Vect$, defined as the full subcategory of $2\Vect^\dagger$ on the $\rmC^*$-2-vector spaces.
In other words: a \emph{$\Z_2^t$-dagger object} in $2\Vect$ is a $\rmC^*$-category, a \emph{$\Z_2^t$-dagger $1$-morphism} a dagger functor, and a $\Z_2^t$-dagger $2$-morphism a unitary natural isomorphism.

\subsubsection{$\Z_2^b\times \Z_2^t$:}

As in the previous paragraph, it is natural to focus on $\Z_2^b\times \Z_2^t$-fixed point structures on a 2-vector space $\Va$ such that the $\Z_2^t$-top part is given by a dagger structure such that all endomorphism algebras are $\rmC^*$-algebras. 
Since the inner product $\overline{\Va}\boxtimes \Va \to \Vect$ is anti-involutive and $\Va$ is equipped with a dagger structure, it makes sense to additionally ask the inner product to be a dagger functor $\overline{\Va}\boxtimes \Va \to \Hilb$. 
{We will therefore define \emph{$\Z_2^b \times \Z_2^t$-dagger $2$-vector spaces} as those $\Z_2^b\times \Z_2^t$-fixed points such that the $\Z_2^t$-fixed point is a $\rmC^*$-category structure with the property that the $\Vect$-inner product is a $\Hilb$-valued dagger functor.}

\begin{ex}
\label{ex:HilbZ/2xZ/2}
Let $\Va = \Hilb$ with its canonical dagger structure viewed as a $\Z_2^t$-dagger object in $2\Vect$.
Since $\Va$ is equivalent to $\Vect$ as a linear category, it follows by Example \ref{ex:vect} that there is a unique $\Vect$-valued inner product on $\mathcal{V}$ given by the functor
\[
\langle -, - \rangle = \mathcal{V}((-)^\dagger, (-)){,}
\]
{where $\dagger$ is the internal dagger.}
There are two possible compatibility data
$\langle (-)^\dagger, (-)^\dagger \rangle \cong \overline{\langle (-), (-)\rangle}^\vee$ making $\mathcal{V}$ into a $\Z_2^b \times \Z_2^t$-fixed point, and each is uniquely determined by its values on $(\C,\C)$.
Only one of the two induces a dagger functor $\overline{\Va} \times \Va \to \Hilb$ such that the exchange isomorphism is unitary.
On a pair $(V,W)$ this gives the Hilbert space of Hilbert--Schmidt maps $V \to W$ with its usual inner product $\langle x,y \rangle = \tr (x^\dagger \circ y)$.
\end{ex}

In~\cite{MR2325696,2411.01678} it is shown that general $\rmW^*$-categories have canonical $\Hilb$-valued inner products and hence in particular $\rmW^*$-categories whose underlying linear category is a 2-vector space are canonically $\Z_2^b\times\Z_2^t$-Hermitian $2$-vector spaces. 
This motivates us to call $\Z_2^b \times \Z_2^t$-dagger $2$-vector spaces \emph{$\rmW^*$-$2$-vector spaces}.
We denote the $\Z_2^b\times\Z_2^t$-dagger 2-category of $\rmW^*$-2-vector spaces and dagger functors by $\rmW^*2\Vect$.\footnote{At this point, the reader may rightfully object that there is no difference in finite-dimensions between a $\rmC^*$ and $\rmW^*$-algebra, and thus there is no difference between finite semisimple $\rmC^*$ and $\rmW^*$-categories.
However, when we look at the 2-categories of (finite-dimensional) $\rmC^*$ and $\rmW^*$-algebras, there is no obvious $\bbZ_2^b$-dagger on $\rmC^*$-Hilbert bimodules, while there is an obvious conjugate Hilbert space bimodule.
In the infinite-dimensional setting, this distinction not only persists, but becomes an actual difference between these types of $*$-algebras.
Indeed, there is no notion of a standard form for an infinite-dimensional $\rmC^*$-algebra, so there is no canonical way to endow a $\rmC^*$-category with a $\Hilb$-valued inner product.
In contrast, the Haagerup standard form exists for any $\rmW^*$-algebra \cite{MR0407615}, which is used to construct the $\Hilb$-valued inner product on a $\rmW^*$-category \cite{MR2325696,2411.01678}.
} 

The $\Z_2^b\times\Z_2^t$-dagger structure manifests itself in the existence of both $\dagger_t$ and $\dagger_b$ operations that commute with each other such that the natural isomorphism $(\dagger_b)^2 \cong \id$ is unitary. This structure is also known as a bi-involutive 2-category~\cite{MR3663592,2411.01678}.

\subsubsection{$SO(2)$:}
Similar as for $\Z_2^b$, we are not aware of a good notion of positive $SO(2)$-Hermitian $2$-vector space. Let us still explain the additional structure we find if we take all of them. Let us denote this category by $\mathsf{CY}2\Vect$. As explained in detail in~\cite[Section 5.1.2]{2307.06485} and~\cite{2504.17764} the additional structure present on 
$\mathsf{CY}2\Vect$ is a pivotal structure, i.e.\ compatible identifications between left and right adjoints of every linear functor, or equivalently, for every 1-morphism $F$ a choice of 1-morphism $F^\vee$, which is equipped with both the structure of a left adjoint and a right adjoint to $F$. In a pivotal $2$-category there is the notion of a pivotal equivalence~\cite[Definition 2.2]{2307.06485}. 
A pivotal equivalence is a 1-morphism $F$ such that $F^\vee$ is an inverse of $F$ whose witnessing 2-isomorphisms to
and from the identity are given in terms of the left and right adjunction data. 

\subsubsection{$O(2)$:}
Recall that an $O(2)$-Hermitian 2-vector space consists of an anti-involutive category $\Va$ and a compatible trace $\tau$. As before, the $O(2)$-dagger objects should at least require the anti-involution to be an actual dagger structure $\dagger$. This allows us to define a Hermitian inner product on all hom spaces 
\begin{align}
\Hom_{\Va}(a,b) \otimes \Hom_{\Va}(a,b) &\longrightarrow \C \\ 
f\otimes g &\longmapsto \tau(f^\dagger \circ g) \ \ .  
\end{align}
It is natural to further restrict to those Hermitian structures for which this inner product is positive definite, recovering the notion of a Baez 2-Hilbert space~\cite{MR1448713}. 
The underlying $\Z_2^b$-Hermitian $2$-vector space of a Baez 2-Hilbert space is the functor
\begin{align}
    \overline{\Va}\boxtimes \Va \xrightarrow{\dagger \boxtimes \id_{\Va}} \Va^{\op}\boxtimes \Va \xrightarrow{\Hom (-,-)} \Vect
\end{align} 
which lifts to a $\Hilb$ valued inner product.
We denote the resulting $O(2)$-dagger 2-category with $1$-morphisms linear dagger functors by $2\Hilb$. 

This category has all the structures discussed above: two functors $\dagger_b\colon 2\Hilb \rightarrow 2\Hilb^{1\op}$ and $\dagger_t \colon 2\Hilb \rightarrow 2\Hilb^{2\op}$ where both functors are the identity on objects and $\dagger_t$ is additionally the identity on 1-morphisms. Furthermore, for every dagger functor, there is a canonical choice of a unitary adjoint functor, going the other direction which is both a left and right adjoint. 

In every $O(2)$-dagger category there is an additional notion of unitary $1$-morphism; those dagger 1-morphisms which are invertible and preserve the full $O(2)$-fixed point data. For 2-Hilbert spaces these are exactly those dagger equivalences which induce unitary isomorphisms on hom Hilbert spaces.

\begin{rem}
In a $\rmW^*$-$2$-vector space $\mathcal{V}$, both the structure of a $\Hilb$-valued inner product and a $\rmC^*$-category are present.
However, unless $\mathcal{V}$ is a $2$-Hilbert space, these two structures need not satisfy
\begin{equation}
\label{eq:2Hilbcondition}
\langle x, y \rangle_{\mathcal{V}} 
= 
\Hom_\cV(x^\dagger,y),
\end{equation} 
and so $\rmW^*$-$2$-vector spaces are a priori more data than the structure of a {$\rmW^*$-category on $\mathcal{V}$}.\footnote{Recall that every $2$-vector space is equivalent to a direct sum of finitely many copies of $\Vect$. It therefore follows by Example \ref{ex:HilbZ/2xZ/2} that every $2$-vector space has a unique up to equivalence $\rmW^*$-$2$-vector space structure, which happens to satisfy \eqref{eq:2Hilbcondition}. However, we have no reason to believe this persists in infinite dimensions.}
Given that more subtle functional-analytic formulas analogous to \eqref{eq:2Hilbcondition} are used in the infinite-dimensional setting (see \cite[Definition 4.1]{2411.01678}), we expect it to be relevant for the study of higher $\rmW^*$-categories.
More specifically, there should be a suitable generalization of $O(2)$-dagger $2$-categories in settings with less dualizability, enriching the existing notion of bi-involutive $2$-categories.
\end{rem}

\section{Formulation in terms of algebras}\label{sec:Alg}
Using the equivalence between $\Alg$ and $2\Vect$, the results of the previous section lead to corresponding notions of unitarity for algebras.
The first part of this section explains what $G$-fixed points in $\Alg$ look like and posits which of those are $G$-dagger, see Table \ref{tab:alg} for the results.
The second part explains the relationship with $G$-fixed points and $G$-dagger structures in $2$-vector spaces.

\subsection{Unitarity for algebras}
Recall that we defined $\Alg$ as the $2$-category of finite-dimensional semisimple complex algebras in which the category of $1$-morphisms $A \to B$ are given by $A$-$B$-bimodules and composition functors are given by the relative tensor product.



\subsubsection{$\Z_2^b$}

One convenient dual functor is given by $(-)^{\op} \colon \Alg \to \Alg^{1\op}$.
It follows that a $\Z_2^b$-fixed point is a ($\C$-antilinear) \emph{stellar structure} in the sense of \cite{1112.1000}, i.e.\ a Morita equivalence $A \cong \ol{A}^{\op}$ equipped with higher $\Z_2^b$-equivariance data. We choose to rewrite the stellar structure as a left $\overline{A}\otimes A$-module $H$ equipped with data making it `nondegenerate and Hermitian'.  
Explicitly the `Hermitian' data is given by an anti-linear involution $J \colon H \to H$ such that $J((\ol{a_1} \otimes a_2) \xi) = (\ol{a_2} \otimes a_1) J(\xi)$ for all $a_1, a_2 \in A$ and $\xi \in H$.
Observe that the left $\overline{A}$-action on $H$ is completely fixed as $(\overline{a}\otimes 1) \xi = J( 1\otimes a )J \xi$.
Because this $\overline{A}$-action needs to combine with the $A$-action to form an $\ol{A} \otimes A$-action, this data satisfies the condition $J a_1 J a_2 \xi = a_2 J a_1 J \xi$ for all $a_1, a_2 \in A$ and $\xi \in H$.
In other words, a stellar structure on $H$ is equivalent to an anti-linear involution $J$ such that $JAJ \subseteq A'$ where $A'$ is the commutant of $A$ in $B(H)$.\footnote{{It then follows from the assumption that $H$ is a Morita equivalence as an $A$-$\ol{A}^{\op}$-bimodule that $JAJ = A'$.}}
We think of $J$ as a \emph{modular conjugation} on $H$.

The $\Z_2^b$-fixed points on invertible $1$-morphisms can be described using compatibility data with $H$.
In case $(H,J)$ comes from a $C^*$-algebra structure on $A$, the $\dagger_b$ of a bimodule reproduces what is known as the conjugate bimodule in the operator algebra literature \cite{MR703809,MR1424954}.
As in $2\Vect$, we will not define $\Z_2^b$-dagger objects in $\Alg$.

\subsubsection{$\Z_2^t$}
\label{Sec: Alg Z2t}
In our conventions, $\Z_2^t$-fixed points agree with $\Z_2^b$-fixed points on objects. 
In particular, a $*$-algebra structure induces a $\Z_2^t$-fixed point structure. We define \emph{$\Z_2^t$-dagger objects in $\Alg$} to be $\rmC^*$-algebras.
A $\Z_2^t$-fixed point structure on a (not necessarily invertible) $1$-morphism $X \colon A \to B$ is given by the data of $X$ intertwining the respective stellar structures on $A$ and $B$.
In the case where $A$ and $B$ are $*$-algebras, this data corresponds to a generalization of a $B$-valued inner product on $X$.
 Namely, a short computation shows the fixed point structure corresponds to a sesquilinear map $\langle - | - \rangle_B \colon \overline{X} \times X \to B$ such that
\begin{equation}
\label{eq:C*hilbim}
\langle x_1|ax_2 \rangle_B = \langle a^* x_1|x_2 \rangle_B  
\qquad 
\langle x_1| x_2 b \rangle_B = \langle x_1| x_2 \rangle_B b
\qquad 
\langle x_1| x_2\rangle_B^* = \langle x_2| x_1\rangle_B.
\end{equation}
{We thus recovered the notion of (not necessarily positive definite) Hilbert $\rmC^*$-bimodule from our formalism.}
{Following operator-algebraic intuition, w}e define a $\Z_2^t$-fixed bimodule between $\rmC^*$-algebras to be \emph{$\Z_2^t$-dagger} if for all $x \in X$, $\langle x|x\rangle_B$ is positive, i.e.\ of the form $b^* b$ for some $b \in B$.
Viewing the $C^*$-algebra also as a $\Z_2^b$-fixed point, observe that the conjugate $B$-$A$-bimodule $X^{\dagger_b}$ of a $\Z_2^t$-dagger bimodule $X$ has a canonical $B$-valued inner product on the other side given by\footnote{The abstract underlying reason for the exchange of sides is that $(-)^{\vee L} \cong (-)^{R \vee}$ in any fully dualizable symmetric monoidal $2$-category.} 
\begin{equation}
\label{eq:wrongside}
    _B \langle x_1, x_2 \rangle = \langle x_2,x_1 \rangle_B = \langle x_1, x_2 \rangle_B^*.
\end{equation}

If $Y$ is a $B$-$C$-bimodule with $C$-valued inner product, if follows by computing the fixed point data on the composition that the inner product on $X \otimes_B Y$ is given by
 \begin{equation}
 \label{eq:CompositionInnerProductDeterined}
 \langle x_1 \otimes y_1| x_2 \otimes y_2 \rangle_C 
 = 
 \langle y_1 | \langle x_1 | x_2 \rangle_B  y_2 \rangle_C.
 \end{equation}
We reiterate that 
in our theory {formulas \eqref{eq:C*hilbim} and \eqref{eq:CompositionInnerProductDeterined} are not a bespoke definition as in \cite{MR703809,MR0367670,MR1303779}}, but rather a consequence of computing $\bbZ_2^t$-fixed point data.
A $\Z_2^t$-fixed $2$-morphism---which is the same as a $\Z_2^t$-dagger $2$-morphism at this categorical level---{can be} shown to be a bimodule map which is unitary in the respective inner products.

\subsubsection{$\Z_2^b\times\Z_2^t$}
\label{Sec: Alg Z2TaB}
A $\Z_2^b \times \Z_2^t$-fixed algebra has two stellar structures with compatibility data. 
As above we ask the $\Z_2^t$-part to correspond to a $\rmC^*$-algebra $A$. Recall that the $\Z_2^b$-part yields a left $\overline{A}\otimes A$-module $H$ equipped with an involution $J \colon H \to \ol{H}$.
Using the $*$-structure of $A$, we may view $H$ as an $A$-$A$-bimodule.
Its right $A$-action is explicitly given by the familiar formula $\xi a = J a^* J \xi$. 

The requirement that the $\Z_2^t$ and $\Z_2^b$ assemble into a $\Z_2^b \times \Z_2^t$-fixed point implies that $H$ comes with $\Z_2^t$-fixed point data as a morphism $ \overline{A}\otimes A \to \C$, i.e.\ comes equipped with a $\C$-valued inner product for which the $ \overline{A}\otimes A$-action is a $*$-action. 
The remaining condition of being a $\Z_2^b \times \Z_2^t$-fixed point is that $J$ is unitary.
It is now natural to define a \emph{$\Z_2^b \times \Z_2^t$-dagger algebra} to be a $\Z_2^b \times \Z_2^t$-fixed point such that the $\Z_2^t$-fixed point data on $H$ is $\Z_2^t$-dagger in the sense defined above, i.e.\ $H$ is a Hilbert space. 

A {n easy argument reducing to the case $A=\C$} shows that the data $(J,H)$ of a $\Z_2^b \times \Z_2^t$-dagger object on a given finite-dimensional $\rmW^*$-algebra $A$ is unique up to isomorphism. 
One especially canonical choice is given by the Haagerup standard form $L^2A$ \cite{MR0407615}, see \cite[IX.1]{MR1943006} for a textbook reference on standard forms.\footnote{Not every choice of $(H,J)$ admits a self dual positive cone $P$ making $(H,J,P)$ into a standard form on the nose. However, in our approach there always exists an isomorphic $(H',J')$ which does have a self dual positive cone satisfying the necessary properties. 
For example, if $A = \C$, $H = \C$ and $J = -\id_\C$, then $(H,J)$ does not admit a self dual positive cone{; the candidate $P = i\R_{\geq 0}$ is not self-dual}. However, $(H,J)$ is isomorphic to $(H,-J)$ under the unitary operator on $H$ given by multiplication by $i$.}
In this finite-dimensional setting, 
by choosing a faithful trace on $A$, we may identify $L^2A=A$ and $J\xi=\xi^*$.\footnote{In fact, identifications of $L^2A$ with $A$ as $A$-$A$ bimodules which identify $J\xi=\xi^*$ and map $1_A$ into the positive cone of $L^2A$ exactly correspond to faithful tracial states on $A$; see \cite[\S3.3]{UQSLbook}}




\begin{rem}
\label{rem:ConnesFusion}
Observe that for a von Neumann algebra $A$, $L^2A$ is the only data needed to bootstrap the $A$-valued inner products on $\rmW^*$-Hilbert $A$-modules $M_A,N_A$ to a $\Hilb$-valued inner product; indeed, we may use the formula
$$
\langle M_A | N_A\rangle
:=
N_A\boxtimes_A L^2A\boxtimes_A M^{\dagger_b}
$$
where the right hand side has the 4-point formula \emph{Connes fusion} inner product \cite{MR1645078}
$$
\langle n_1\boxtimes \xi_1 \boxtimes \overline{m_1}|n_2\boxtimes \xi_2\boxtimes \overline{m_2}\rangle
:=
\langle \xi_1 | \langle n_1|n_2\rangle_A\cdot \xi_2\cdot {}_A\langle \overline{m_2}, \overline{m_1}\rangle \rangle_{L^2A}
$$
where ${}_A\langle\overline{m_2}, \overline{m_1}\rangle:= \langle m_1|m_2\rangle_A$ 
is the left $A$-valued inner product on $M^{\dagger_b}$ discussed in \eqref{eq:wrongside}.
We will see this formula emerge directly from $\bbZ_2^b\times \bbZ_2^t$-fixed point data in \eqref{eq:4PointFormula} below.

Conversely, the $\Hilb$-valued inner product on the $\rmW^*$-category of $\rmW^*$-Hilbert $A$-modules reconstructs $L^2A$ from $\langle A|A\rangle^{\Hilb}$ \cite{2411.01678}.
\end{rem}



\subsubsection{$SO(2)$}
The $SO(2)$-action on the core of $\Alg$ is given by sending $A$ to the $A$-$A$-bimodule $A^* := \Hom_\C(A,\C)$~\cite{oritthesis, MR3931946}.
Therefore an $SO(2)$-fixed point is an $A$-$A$-bimodule isomorphism $A \cong A^*$, which is exactly a symmetric Frobenius structure $\tau \colon A \to \C$. The corresponding $SO(2)$-dagger structure (i.e.\ pivotal structure) on the bicategory $\Alg$ is for example discussed in~\cite{MR2742426}.

\subsubsection{$O(2)$}
An $O(2)$-fixed point is a stellar Frobenius algebra $A$; a stellar algebra with a compatible symmetric Frobenius structure, see \cite{MR4881921}.
If $A$ is a $*$-algebra with symmetric Frobenius structure $\tau$ this compatibility condition is $\tau(a^*) = \ol{\tau(a)}$.
In that case $(a,b) := \tau(a^* b)$ defines a nondegenerate Hermitian inner product on $A$.
We define the algebra to be \emph{$O(2)$-dagger} when this inner product is positive, i.e.\ when $\tau$ is a faithful trace in the $*$-algebraic sense. 
This $O(2)$-fixed point data is exactly the structure of an \emph{$\rmH^*$-algebra} on $A$ \cite{MR13235,MR3971584}.

Given two $H^*$-algebras, we can also use $\tau$ to build the $\C$-valued Hermitian inner product $(x_1,x_2) := \tau (\langle x_1, x_2 \rangle)$ on any $\Z_2^t$-fixed bimodule $X$.
Moreover, $X$ will be a Hilbert space if the algebra valued inner product was positive since $\tau(a^* a) = |\tau(a)|^2 \geq 0$ is a faithful positive trace.
In this Hilbert space both $A$ and $B$ act unitarily. 
Conversely for any Hilbert space structure $(-,-)$ on $X$ for which $A$ and $B$ act unitarily, there is a unique $B$-valued inner product $\langle-,-\rangle$ such that $(x_1,x_2) = \tau (\langle x_1, x_2 \rangle)$ by nondegeneracy of $\tau$.
There is a conceptually analogous but functional-analytically more delicate construction that connects $\rmC^*$-correspondences with Hilbert space bimodules for $\rmW^*$-algebras~\cite{MR945550}.

\begin{table}
    \centering
    \begin{small}
    \begin{tabular}{c||p{33mm} |p{26mm} |p{36mm} |p{34mm} }
        $G \leq O(2)$ & $G$-fixed object & $G$-dagger {object}  & $G$-fixed {1-morphism} & $G$-dagger 1-morphism  \\ 
        \hline \hline
        $\Z_2^b$ & stellar algebra & -  & Morita equivalence \mbox{compatible} with stellar structure  & - \\
        \hline
        $\Z_2^t$ &  stellar algebra & $\rmC^*$-algebra & bimodule with \mbox{Hermitian} form & $\rmC^*$ Hilbert bimodule \\
        \hline
        $\Z_2^b\times \Z_2^t$ & 2 compatible stellar algebra structures  & $\rmW^*$-algebra  with standard form $L^2(A)$ & compatible with stellar structure \& Hermitian form & $\rmC^*$ Hilbert bimodule \\
        \hline
        $SO(2)$ & symmetric Frobenius  & - & trace-preserving Morita equivalence & - \\ 
        \hline
        $O(2)$ & stellar Frobenius & $\rmH^*$-algebras & trace-preserving Morita equivalence with \mbox{Hermitian} form & Trace-preserving Hilbert $\rmC^*$-Morita equivalence \\
    \end{tabular}
    \end{small}
    \caption{$G$-fixed points as well as $G$-dagger structures on objects and $1$-morphisms in the $2$-category of algebras, bimodules and bimodule maps.}
    \label{tab:alg}
\end{table}

\subsubsection{Relationship with $2\Vect$}

The equivalence $\Alg \to 2\Vect$ is given by sending $A$ to the finite semisimple category $\Mod(A)$ of finite-dimensional right $A$-modules, an $A$-$B$-bimodule $X$ to the $\C$-linear functor $(-) \otimes_A X \colon \Mod(A) \to \Mod(B)$ and a bimodule map to the obvious natural transformation. 

If $A$ is a stellar algebra interpreted at a $\Z_2^b$-fixed point, then $\Mod(A)$ has the $\Vect$-valued inner product
\begin{align}
\label{Eq: pairing}
     M_A, N_A \longmapsto (\overline{M}\otimes N) \otimes_{\overline{A}\otimes A} H \in \Vect.
\end{align}
Here, $H$ is the $\ol{A}\otimes A$-module corresponding to the stellar structure. 
The unitary isomorphism $\langle M| N \rangle \cong \overline{\langle N| M \rangle}$ 
corresponds to the anti-unitary involution
$\langle M| N \rangle \to \langle N| M \rangle$ 
given by 
\[
\ol{m} \otimes n \otimes_{\ol{A}\otimes A} \xi \mapsto \ol{n} \otimes m \otimes_{\ol{A}\otimes A} J \xi.
\]
To see this map is well-defined, observe that since $\overline{a}\xi=JaJ\xi$, we have
\[
\begin{tikzcd}
    \ol{m a} \otimes nb \otimes \xi 
    \ar[r, mapsto] \ar[d,equal]
    & 
    \ol{nb} \otimes ma \otimes J \xi 
    \ar[d, equal] 
    \\
    \ol{m} \otimes n\otimes (\ol{a}\otimes b) \xi 
    \ar[d, equal] 
    &\ar[d, equal]
    \ol{n}\ol{b} \otimes ma \otimes J \xi 
    \\
    \ol{m} \otimes n \otimes JaJ b \xi 
    \ar[r, mapsto] 
    &
    \ol{n} \otimes m \otimes aJbJ(J\xi) 
\end{tikzcd}
\]
In this way $\langle A,A\rangle \cong \ol{\langle A,A\rangle}$ recovers $J \colon H \to \overline{H}$ from the $\Vect$-valued inner product.
An invertible $A$-$B$-bimodule $X$ intertwines stellar structures on $A$ and $B$ exactly when the corresponding functor $(-) \otimes_A X$ is inner-product preserving in the sense of \eqref{eq:bottomunitary}.

If the stellar structure is interpreted as a $\Z_2^t$-fixed point, then $\Mod(A)$ is an anti-involutive category via $M \mapsto \Hom_A(M,A)$.
Here we use the stellar structure to make the obvious left $A$-module structure on $\Hom_A(M,A)$ into a right $A$-module.
In case the stellar structure on $A$ comes from a $*$-algebra, a fixed point for this anti-involution is the same as an $A$-valued inner product as discussed above.
In other words, a fixed point structure on an $A$-module $M$ as an object of $\Mod(A)$ is the same as a $\Z_2^t$-fixed point structure on the corresponding $\C$-$A$-bimodule seen as a $1$-morphism in $\Alg$. 
To build from $\Mod(A)$ a dagger category in case $A$ is a $\rmC^*$-algebra, we choose those fixed points for which $\langle m|m\rangle$ is of the form $a^* a$ for all $a \in A$.\footnote{Unless $A$ is a $\rmC^*$-algebra, there is no assurance that the sum of two elements of the form $a^*a$ will again be of this form.}
This makes $\Mod(A)$ into a $\rmC^*$-category.


If $A,B$ are $\Z_2^t$-fixed algebras and $X$ is a $\Z_2^t$-fixed $A$-$B$-bimodule, then $(-) \otimes_A X \colon \Mod(A) \to \Mod(B)$ becomes an anti-involutive functor because the $\Z_2^t$-fixed point structure on the $\C$-$A$-bimodule (right $A$-module) $M$ composes with $X$ to obtain a $\Z_2^t$-fixed point structure on $M \otimes_A X$.
Since $\Z_2^t$-dagger structures on bimodules compose, $(-) \otimes_A X$ is a dagger functor if the $B$-valued Hermitian inner product on $X$ is positive.

If $(A,H,J)$ is a $\Z_2^b \times \Z_2^t$-dagger algebra, then it follows from the requirement that $H$ is a Hilbert space that the corresponding $\Vect$-valued inner product \eqref{Eq: pairing} is $\Hilb$-valued.
Indeed, if $M$ and $N$ are equipped with positive $A$-valued inner products, then the Hermitian pairing on $\langle M|N\rangle$ induced by \eqref{Eq: pairing} being an anti-involutive functor is given by the composition of the $\overline{A}\otimes A$-Hilbert module structures on $\ol{M} \otimes N$ and $H$.
This defines a Hilbert space as the relative tensor product of a Hilbert module with a Hilbert space is again a Hilbert space.
Explicitly working out the formula for the inner product on the Hilbert space $\langle M|N\rangle$ using Equation \eqref{eq:CompositionInnerProductDeterined} gives
\begin{equation}
\begin{aligned}
\langle \ol{m_1} \otimes n_1 \otimes \xi_1| \ol{m_2} \otimes n_2 \otimes \xi_2 \rangle 
&= \langle \xi_1| \langle \ol{m_1} \otimes n_1| \ol{m_2} \otimes n_2 \rangle_{\overline{A} \otimes A} \xi_2 \rangle
\\
&= \langle \xi_1| (\langle \ol{m_1}| \ol{m_2} \rangle_{\overline{A}} \otimes \langle n_1| n_2 \rangle_{A}) \xi_2 \rangle
\\
&= 
\langle \xi_1| (\overline{\langle m_1| m_2 \rangle_{A}} \otimes \langle n_1| n_2 \rangle_{A}) \xi_2 \rangle
\\&=
\langle \xi_1| J\langle m_1| m_2 \rangle_{A}^*J 
\langle n_1| n_2 \rangle_{A} \cdot \xi_2 \rangle
\\&=
\langle \xi_1|
\langle n_1| n_2 \rangle_{A} \cdot \xi_2 \cdot {}_A\langle \overline{m_2} , \overline{m_1}\rangle\rangle
\end{aligned}
\label{eq:4PointFormula}
\end{equation}
where in the last equality above, we used the left $A$-valued inner product on the left $A$-module $ M^{\dagger_b}$.
When $H=L^2A$, we may identify
$$
\langle M|N\rangle = N\otimes_A L^2A \otimes_A M^{\dagger_b},
$$
and \eqref{eq:4PointFormula} above exactly reproduces the 4-point formula for \emph{Connes fusion}~\cite{MR1645078} as in Remark \ref{rem:ConnesFusion}.
In particular, $\langle A,A \rangle = L^2 A$ as a Hilbert space.

Given a Frobenius structure $\tau$, the Calabi-Yau structure on $\Mod(A)$ is given as follows.
First note that because every finite-dimensional $A$-module $M$ is finitely generated and projective, the map 
\[
M \otimes_\C \Hom_A(M, A) \to \End M_A
\]
is an isomorphism of vector spaces.
We thus get a trace on $\End M_A$ by 
\[
M \otimes_\C \Hom_A(M, A) \to M \otimes_A \Hom_A(M, A) \xrightarrow{\ev} A \xrightarrow{\tau} \C.
\]
Now $A$ is an $\rmH^*$-algebra if and only if the Calabi-Yau and dagger structures on $\Mod(A)$ we defined make it into a $2$-Hilbert space.

Note that every $\rmH^*$-algebra $A$ is a $\rmC^*$-algebra since 
$A$ sits unitally inside $B(A)$ (which is finite dimensional) as a $*$-closed subalgebra.
This corresponds to the fact that every $2$-Hilbert space is automatically a $\rmC^*$-category.

Conversely, if $\mathcal{V}$ is a $2$-vector space with certain Hermitian structure (anti-involutive/vector space valued inner product, $C^*$-category, Calabi-Yau, $2$-Hilbert), it is easy to extract the corresponding Hermitian structure (stellar, $C^*$, symmetric Frobenius, $\rmH^*$) on the algebra $\End_\mathcal{V}(x)$, where $x\in \mathcal{V}$ is a generator.
We conclude:

\begin{thm}
For $G = \Z_2^t$, $\Z_2^b \times \Z_2^t$, and $O(2)$ we have that $G$-dagger objects in $\Alg$ as explained here and $G$-dagger objects in $2\Vect$ as in Section \S\ref{sec:2Hilb} correspond to one another. 
\end{thm}

\section{An inductive procedure to select positive Hermitian morphisms}\label{sec:Inductive}
As explained in
Section \S\ref{sec:2Hilb}, there are established notions of positivity commonly used in the literature in the setting of 2-vector spaces. 
The situation becomes scarcer for higher $n$; for $n=3$ see the recent paper~\cite{2410.05120}. 
In this section, we give a proposal for an inductive way of selecting positive Hermitian structures that reproduces the above choices for $n=2$, and has the potential to generalize to arbitrary dimensions. 

Let us start by explaining our approach for $1$-vector spaces.  
Let $(V,h\colon V \to \ol{V}^\vee)$ be a Hermitian vector space and let $f\colon \C \to V$ be a linear map. We can consider the endomorphism $f^\dagger \circ f$ of $\C$ which can be identified with a complex number. From
\[
\overline{( f^\dagger \circ f)}= ( f^\dagger \circ f)^\dagger = f^\dagger \circ f^{\dagger \dagger } = f^\dagger \circ f 
\] 
it follows that $f^\dagger \circ f$ is a real number. 
Observe that $(V,h)$ is a Hilbert space if and only if, for every linear map $f: \bbC \to V$, the real number $f^\dagger f$ is non-negative.
Indeed, $f$ corresponds to an element $v \in V$ and $f^\dagger f(1) = \langle v,v \rangle$ corresponding to the Hermitian form $h$.
In this way, Hilbert spaces can be defined from knowing what $O(1)$-fixed points and positive real numbers are. 

Abstractly the definition of positivity we just discussed only relies on the notion of positivity for the endomorphisms of the monoidal unit,\footnote{Studying this in more general categories would be interesting, especially when $1$ is not a generator, such as in the category of super vector spaces.} and so it is suitable for categorification purposes. The rough strategy for extending this to 2-vector spaces is analogous: 
We want to define a $G$-fixed 2-vector space $\Va$ to be `$G$-Hilbert' if for enough morphisms $F\colon \Vect \to \Va $ with certain fixed point data, the induced fixed point $F^{\dagger_b} \circ F \in \End(\Vect)\cong \Vect$ is `Hilbert'. 
For us to carry out this approach, the subgroup $G$ must contain $\Z_2^b$ to allow us to define $F^{\dagger_b}$, and must contain $\Z_2^t$ to allow us to obtain fixed point data on $F^{\dagger_b} \circ F$, which we can then require to be a Hilbert space. 
Therefore of the subgroups we discussed, only the cases $\Z_2^b\times \Z_2^t$ and $O(2)$ are suitable. 

Before continuing, first recall that $\Vect \in 2\Vect$ has canonical $O(2)$-fixed point structure (and hence also a $\Z_2^b\times \Z_2^t$-fixed point structure) induced by the fact that it is the monoidal unit:

\begin{lem}
\label{lem:VectO2herm}
The canonical $O(2)$-Hermitian structure on the $2$-vector space $\Vect$ is given as follows. 
\begin{enumerate}
    \item The $\Z_2^t$-Hermitian structure is the anti-involution $d = \ol{(-)}^\vee$ with the canonical evaluation map $V \to \ol{\ol{V}^\vee}^\vee$
    \item The $\Z_2^b$-Hermitian structure is given by $\langle V,W \rangle_{\Vect} = \Vect(\overline{V}^\vee,W) = \overline{V} \otimes W$ with the canonical isomorphism $\overline{V} \otimes W \cong \overline{\overline{W} \otimes V}$
    \item The $SO(2)$-Hermitian structure is the standard trace.
    \item The compatibility isomorphism $\langle dV,dW \rangle_{\Vect} \cong \overline{\langle V,W \rangle_{\Vect}^\vee}$ is induced by the canonical isomorphism $V^\vee \otimes W^\vee \cong (V \otimes W)^\vee$
\end{enumerate}
\end{lem}

\subsection{\texorpdfstring{$\Z_2^b\times \Z_2^t$}{Z2xZ2}-dagger objects revisited}

Let $\mathcal{V}$ be a $\Z_2^b\times \Z_2^t$-Hermitian $2$-vector space.
Let $F\colon \Vect \to \mathcal{V}$ be a $1$-morphism of $2$-vector spaces.
This functor is equivalent to the data of a single object $F(\C) \in \mathcal{V}$ by additivity of $F$.
Recall that $F^{\dagger_b}$ is defined as the composition
\[
F^{\dagger_b}\colon \mathcal{V} \xrightarrow{D} \ol{\mathcal{V}}^\vee \xrightarrow{\ol{F}^\vee} \ol{\Vect}^\vee \cong \Vect,
\]
where $D$ is the $\Z_2^b$-fixed point data on $\mathcal{V}$.
This implies that there is a natural isomorphism
\[
\langle F(W), x \rangle_{\mathcal{V}} \cong\langle W, F^{\dagger_b}(x) \rangle_{\Vect} = \overline{W} \otimes F^{\dagger_b}(x)  \quad x \in \mathcal{V}, W \in \Vect
\]
of functors $\overline{\Vect} \times \mathcal{V} \to \Vect$.

Recall that $\bbZ^t_2$-fixed point data on $F$ is equivalent to making $F$ an anti-involutive functor.
In the case at hand, we can relate this to $\Z_2$-fixed point data in the anti-involutive category $\mathcal{V}$ itself:

\begin{lem}
Let $\mathcal{V}$ be a $\bbZ^t_2$-Hermitian $2$-vector space.
The equivalence $2\Vect(\Vect, \mathcal{V}) \to \mathcal{V}$ 
of $2$-vector spaces given by $F \mapsto F(\C)$ is an equivalence of $\Z_2^t$-fixed points, i.e.\ an isomorphism in $(\iota_1 2\Vect)^{\Z_2^t}$.
\end{lem}

Here $2\Vect(\Vect, \mathcal{V})$ is the internal hom in $2\Vect$, which is a $\bbZ^t_2$-fixed point using the given $\bbZ^t_2$-fixed point structure on $\mathcal{V}$ and the canonical $\bbZ^t_2$-fixed point structure on $\Vect$.
This means that the anti-involution on the category $2\Vect(\Vect, \mathcal{V})$ is chosen so that $\Z_2$-fixed points are anti-involutive functors.
The lemma implies that $\iota_0(2\Vect(\Vect, \mathcal{V})) \to \iota_0 \mathcal{V}$ is $\Z_2$-equivariant.
Therefore the following are equivalent for a $1$-morphism $F \colon \Vect \to \mathcal{V}$:
\begin{enumerate}
    \item An anti-involutive structure on $F$.
    \item Giving $c:=F(\C) \in \mathcal{V}$ a $\Z_2$-fixed point structure with respect to the anti-involution on $\mathcal{V}$.
\end{enumerate}

The compatibility between $\Z_2^b$ and $\Z_2^t$ implies that if $F \colon \Vect \to \mathcal{V}$ is a $\bbZ^t_2$-fixed $1$-morphism into a $\Z_2^b \times \Z_2^t$-Hermitian $2$-vector space, the morphism $F^{\dagger_b} F$ also comes with canonical $\Z_2^t$-fixed point data.
In particular, by the above lemma we obtain a nondegenerate Hermitian form on the vector space 
\[
F^{\dagger_b}F(\bbC) \cong \langle \C, F^{\dagger_b} F (\C) \rangle_{\Vect} \cong \langle F(\C), F(\C) \rangle_{\mathcal{V}} = \langle c,c \rangle_{\mathcal{V}}.
\]
We think of the Hermitian form as the data of $F^{\dagger_b} F$ (or $\langle c,c \rangle_{\mathcal{V}}$) being `self-adjoint'. 
The $\Z_2^t$-fixed point structure on $F^{\dagger_b}F$ is described by the diagram
\begin{equation}\label{Eq: Herm} 
    \begin{tikzcd}
        \Vect \ar[r,"F"] \ar[d,"{\overline{(-)}^\vee}"] & \mathcal{V} \ar[d,"d"] \ar[r,"D"] \ar[rrr,"F^{\dagger_b}", bend left] & \overline{\mathcal{V}}^\vee \ar[r,"\overline{F}^\vee"] & \overline{\Vect}^\vee & \Vect \ar[d,"{\overline{(-)}^\vee}"] \ar[l,"\sim"]
        \\
        \overline{\Vect}^{\op} \ar[r,"{\overline{F}^{\op}}"] & \overline{\mathcal{V}}^{\op} \ar[r,"{\overline{D}^\vee}"] & \overline{\overline{\mathcal{V}}^{\op}}^\vee \ar[u,"{\overline{d}^\vee}"] \ar[r,"{\ol{\overline{F}^\vee}^{\op}}"] & \overline{\overline{\Vect}^{\op}}^\vee \ar[u]  & \overline{\Vect}^{\op} \ar[l,"\sim"]
    \end{tikzcd}
\end{equation}
    Note that, as in Example \ref{sec:2Herm} we are implicitly identifying $(\mathcal{V}^{\op})^\vee \cong (\mathcal{V}^\vee)^{\op}$ using the dual on $\Vect$.
    From left to right the commutation data is as follows:
    \begin{itemize}
        \item The first square corresponds to the $\Z_2^t$-fixed point structure on $F$.
        \item The second square is the compatibility of the $\Z_2^t$-Hermitian and $\Z_2^b$-Hermitian $2$-vector space structures on $\mathcal{V}$.
        \item The third square is the functor $\overline{(-)}^{\vee}$ applied to the $\Z_2^t$ fixed point structure on $F$
        \item The fourth square is the canonical $\Z_2^b \times \Z_2^t$-Hermitian $2$-vector space structure on $\Vect$ (Lemma \ref{lem:VectO2herm}).
    \end{itemize}

\begin{lem}
\label{lem:categorifyherm}
    Let $h_c: c \to dc$ be a $\Z_2$-fixed point corresponding to an anti-involutive structure on $F \colon \Vect \to \mathcal{V}$.
The data of $F^{\dagger_b} F: \Vect \to \Vect$ being anti-involutive corresponds to the Hermitian form on $\langle c, c \rangle \in \Vect$ given by
\[
\langle c,c \rangle \cong \langle dc, dc \rangle \cong \overline{\langle c,c \rangle}^\vee,
\]
where the first isomorphism comes from $h$ and the second isomorphism is \eqref{eq:aiinnerproduct}.
\end{lem}
\begin{proof}
Writing out the data in Diagram~\eqref{Eq: Herm} gives the result.
\end{proof}

\subsection{\texorpdfstring{$O(2)$}{O(2)}-dagger objects revisited}

In order to connect to $2$-Hilbert spaces, we now focus on the case where $\mathcal{V}$ is an $O(2)$-fixed point.
Recall from Example \ref{sec:2Herm} that the categorified inner product on $\mathcal{V}$ is then given in terms of the anti-involution as 
\[
\ol{\mathcal{V}} \boxtimes \mathcal{V} \to \Vect \quad (x,y) \mapsto \Hom_{\mathcal{V}}(dx,y).
\]
We thus see that $F^{\dagger_b} F: \Vect \to \Vect$ is given by tensoring with $\Hom(dc,c)$.

We describe the corresponding data explicitly in the following Lemma:

\begin{lem}
Let $h_c: c \to dc$ be a $\Z_2$-fixed point corresponding to an anti-involutive structure on $F$.
The data of $F^{\dagger_b} F: \Vect \to \Vect$ being anti-involutive corresponds to the Hermitian form on $\Hom(dc,c) \in \Vect$ given by
\[
\Hom(dc,c) \xrightarrow{h_c} \Hom(c,dc) \xrightarrow{d} \overline{\Hom(d^2 c, dc)} \xrightarrow{\eta} \overline{\Hom(c,dc)} \xrightarrow{\kappa} \ol{\Hom(dc,c)}^\vee
\]
where the first isomorphism both composes and precomposes with $h_c$, and $\kappa$ is induced by the Calabi-Yau structure on $\mathcal{V}$.
\end{lem}
\begin{proof}
Consequence of Lemma \ref{lem:categorifyherm} using Equations \eqref{eq:Z2xZ2} and \eqref{eq:cathermitian} for the $\Z_2^b\times\Z_2^t$-fixed point data on $\mathcal{V}$ induced by its $O(2)$-fixed point data.
\end{proof}
We can now define $F^{\dagger_b}F$ to be `positive' when the Hermitian form on $F^{\dagger_b}F(\C)$ is positive definite. Part of the definition of $\Z_2^t\times \Z_2^b$- and $O(2)$-dagger objects is a lift of the $\Z_2^t$-volution to a dagger structure. This lift is the same as the choice of $\Z_2$-fixed point data on each object of $\Va$. The previous discussion gives us a canonical way of picking those leading to the following definition. 
\begin{defn}
\label{def:G2Hilb}
Let $G$ be $\Z_2^b \times \Z^t_2$ or $O(2)$. A \emph{$G$ $2$-Hilbert space} is a $2$-vector space $\Va$ together with a $G$-fixed point structure and a collection of $\Z_2$-fixed point structures on every object of $\Va$ such that the corresponding Hermitian vector space $F^{\dagger_b}\circ F$ is a Hilbert space. 
\end{defn}

This reproduces the definitions of Section~\ref{sec:2Hilb} in a conceptual way. In the case of 2-Hilbert spaces, this choice turns out not to be any additional information: 
\begin{cor}
    Let $\mathcal{V}$ be an $O(2)$-Hermitian $2$-vector space.  
    Then $\mathcal{V}$ is a $2$-Hilbert space if and only if for every object $c \in \mathcal{V}$ there exists $\bbZ_2^t$-fixed point data on the corresponding linear functor $F\colon \Vect \to \mathcal{V}$ such that the Hermitian vector space $F^{\dagger_b} F(\bbC)$ is a Hilbert space.  
\end{cor}

\begin{rem}
    We expect there is a criterion analogous to Definition \ref{def:G2Hilb} that selects $G$-dagger $1$-morphisms in $2\Vect$ in agreement with Section~\ref{sec:2Hilb}.
It would be interesting to work this out in more detail, and study its potential generalization to dagger $k$-morphisms in $n\Vect$.
\end{rem}

\subsection{\texorpdfstring{$O(n)$-dagger $n$-vector spaces}{O(n)-dagger n-vector spaces}}

We conclude this section by speculating about generalizations to higher categories, which we expect to lead to a satisfactory definition of $n$-Hilbert spaces. 
There is an equivalence $\End_{n\Vect} (\mathbbm{1}_{n\Vect}) \cong (n-1)\Vect$ of $(n-1)$-categories where $\mathbbm{1}_{n\Vect} \in n\Vect$ is the monoidal unit.
Moreover, the $O(n-1)$-action on $\iota_1 n\Vect$ gets mapped to the $O(n-1)$-action on $\iota_0 (n-1)\Vect$ under this identification.
By induction, we know how to lift $O(n-1)$-fixed points for the latter action to $(n-1)$-Hilbert spaces, and therefore we can identify which $O(n-1)$-fixed $1$-morphisms $\mathbbm{1}_{n\Vect} \to \mathbbm{1}_{n\Vect}$ are given by `tensoring with an $(n-1)$-Hilbert space'.\footnote{When everything is in place, we expect this to lead to an equivalence $\End_{n\Hilb}(\mathbbm{1}_{n\Hilb}) \cong (n-1)\Hilb$ of $O(n-1)$-dagger categories.}

This observation allows us to define inductively the notion of $n$-Hilbert space following the reasoning above. 
Let $\mathsf{C} \in n\Vect$ be an $O(n)$-fixed point of which we want to know whether it is $n$-Hilbert. 
Let $F \colon \mathbbm{1}_{n\Vect}\to \mathsf{C}$ be a $1$-morphism with $O(n-1)$-fixed point data.
The remaining $\Z_2^b$-fixed point data on $\mathsf{C}$ allows us to define $F^{\dagger_b} \colon \mathsf{C} \to \mathbbm{1}_{n\Vect}$.
We then obtain the $O(n-1)$-fixed point $F^{\dagger_b} F \in \End(\mathbbm{1})$, which corresponds to an $O(n-1)$-fixed point in $(n-1)$-vector spaces.
Concretely, we think of $F$ as an element $c \in \mathsf{C}$ and we want to know whether $\langle c,c \rangle_{\mathsf{C}} \in (n-1)\Vect$ is an $(n-1)$-Hilbert space, where $\langle-,-\rangle_{\mathsf{C}}$ denotes the $(n-1)$-vector space valued inner product on $\mathsf{C}$.
Now we can define an $n$-Hilbert space to be an $O(n)$-fixed point in $n$-vector spaces together with for every $1$-morphism $F\colon 1 \to \mathsf{C}$ the data of an $O(n-1)$-fixed point and a lift of $F^{\dagger_b} F$ to an $(n-1)$-Hilbert space. This will inductively impose $O(n-k)$ Hilbert space structures on the $k$-morphisms in $\mathsf{C}$ as well.

{Observe that: 
\begin{enumerate}
    \item for $n=1$, it is a condition for an $O(1)$-fixed point to be a Hilbert space;
    \item for $n=2$, the data of lifting an $O(2)$-fixed point $\mathcal{V}$ to a $2$-Hilbert space is given by $O(1)$-fixed points $F \colon 1 \to \mathcal{V}$ for all objects of $\mathcal{V}$ satisfying a condition that a certain $O(1)$-fixed $1$-vector space is a Hilbert space;
    \item for $n=3$, the $O(2)$-fixed point data on $F \colon 1 \to \mathfrak{V}$ gives an $O(2)$-fixed $2$-vector space and hence it is yet further data to ask it to be $2$-Hilbert space.
\end{enumerate}
In the third case we conjecture:}

\begin{conj}
\label{conj:O3FixedPoint}
An $O(3)$-fixed point $\fX\in 3\Vect$ 
equipped with an $O(2)$-fixed point structure on every $F: 2\Vect\to \fX$ and a lift of $F^{\dag_b}F$ to a 2-Hilbert space 
exactly corresponds to a 3-Hilbert space as defined in \cite{2410.05120}.
\end{conj}

\bibliographystyle{alpha}
{\footnotesize{
\bibliography{bibliography}
}}
\end{document}